\documentclass[12pt]{amsart}
\usepackage{amsmath,amssymb,amsfonts}
\usepackage{enumerate}
\usepackage{anysize}
\usepackage{amscd, amssymb, latexsym, amsmath, amscd}
\usepackage[all]{xy}
\usepackage{pb-diagram}
\input xy
\xyoption{all}
\usepackage{pb-diagram}
\usepackage[all]{xy}
\input xy
\xyoption{all}

\newtheorem{theorem}{Theorem}[section]

\newtheorem{lem}[theorem]{Lemma}
\newtheorem{cor}[theorem]{Corollary}

\newtheorem{prop}[theorem]{Proposition}

\newtheorem{conj}[theorem]{Conjecture}

\newcommand{\A}{\mathbb{A}}

\newcommand{\Z}{\mathbb{Z}}

\newcommand{\C}{\mathbb{C}}
\newcommand{\F}{\mathbb{F}}
\newcommand{\N}{\mathbb{N}}

\newcommand{\Hom}{{\rm Hom}\,}

\def\M{{\rm M}}

\def\SL{{\rm SL}}

\def\Sp{{\rm Sp}}

\def\GO{{\rm GO}}
\def\GL{{\rm GL}}
\def\PGL{{\rm PGL}}
\def\PD{{\rm PD}}
\def\GSO{{\rm GSO}}

\def\SO{{\rm SO}}
\def\O{{\mathcal O}}
\def\OR{{\rm O}}

\begin{document}

\title[A local-global question in automorphic forms]
{A local-global question in automorphic forms}
\author{U. K. Anandavardhanan}
\email{anand@math.iitb.ac.in}
\address{Department of Mathematics, Indian Institute of Technology Bombay, Mumbai - 400 076, INDIA.}

\author{Dipendra Prasad}
\email{dprasad@math.tifr.res.in}
\address{Tata Institute of Fundamental Research, Homi Bhabha Road, Mumbai - 400 005, INDIA.}

\subjclass{Primary 11F70; Secondary 22E55}

\date{}

\begin{abstract}

In this paper, we consider the $\SL(2)$ analogue of two well-known theorems about period
integrals of automorphic forms on $\GL(2)$: one due to Harder-Langlands-Rapoport, and the other due 
to Waldspurger.

\end{abstract}

\maketitle

\setcounter{tocdepth}{1}

\tableofcontents

\section{Introduction}

Let $F$ be a number field and ${\Bbb A}_F$ its ad\`ele ring. Let $G$ be
a reductive algebraic group over $F$ with center $Z$, and $H$ a reductive subgroup of $G$
over $F$ containing $Z$. For an automorphic form $\phi$ on $G({\Bbb A}_F)$ on which $Z({\Bbb A}_F)$ acts trivially, the period integral of $\phi$ with respect to $H$ is defined to be the integral (when convergent, which is the case  if $\phi$ is cuspidal 
and $H(F)Z({\Bbb A}_F) \backslash H({\Bbb A}_F)$ has finite volume) $${\mathcal P}(\phi)=\int_{H(F)Z({\Bbb A}_F) \backslash H({\Bbb A}_F)}\phi(h)dh,$$
where $dh$ is the natural measure on 
$H(F)Z({\Bbb A}_F) \backslash H({\Bbb A}_F)$.  

An automorphic representation $\Pi$
of $G({\Bbb A}_F)$ is said to be globally distinguished with respect
to $H$ if this period integral
is nonzero for some $\phi \in \Pi$. More
generally, if $\chi$ is a one-dimensional representation of
$H({\Bbb A}_F)$ trivial on $H(F)$ such that 
$Z({\Bbb A}_F)$ acts trivially on
$\phi(h)\chi^{-1}(h)$, and 
$$\int_{H(F)Z({\Bbb A}_F)\backslash H({\Bbb A}_F)}\phi(h)\chi^{-1}(h)dh,$$
is nonzero for some $\phi \in \Pi$, then $\Pi$ is said to be 
$\chi$-distinguished.

The corresponding local notion is defined as follows. If $\Pi_v$ is an irreducible admissible 
representation of $G(F_v)$, $\Pi_v$ is said to be locally distinguished with respect to 
$H(F_v)$ if it admits a non-trivial $H(F_v)$-invariant linear form. Distinction with respect to 
$\chi_v$, a character of $H(F_v)$,  is defined in a similar manner.

It is obvious that if $\Pi=\otimes_v \Pi_v$ is globally distinguished 
with respect to $H({\mathbb A}_F)$, then each $\Pi_v$ is 
locally distinguished with respect to $H(F_v)$. Indeed, the period 
integral `restricted'  to $\Pi_v$ is a non-trivial $H(F_v)$-invariant 
linear form. The local-global question asks the converse: if $\Pi$ 
is such that each $\Pi_v$ is locally distinguished, 
is $\Pi$ globally distinguished? It seems 
best to break this question into two parts.

\vspace{4mm}

{\bf Question 1:} Let $\Pi=\otimes_v \Pi_v$ be a cuspidal
 automorphic representation
of $G(\A_F)$ such that each of the representations $\Pi_v$ of $G(F_v)$ 
is distinguished by $H(F_v)$. Is there an automorphic representation, say $\Pi'$, 
in the global $L$-packet of $G(\A_F)$
determined by $\Pi $ which is globally distinguished by $H(\A_F)$? 

\vspace{4mm}

Naturally, since the question is about an $L$-packet, one might expect 
some $L$-functions to intervene in the answer to this question. This is the 
case in the work of Waldspurger \cite{waldspurger}, as generalized in \cite{gross1}, and \cite{gross2}, where one 
does not need to go far to look for $\Pi'$, and in fact it must be $\Pi$
itself --- because of a strong multiplicity one theorem known in this context for the 
whole $L$-packet --- if a certain central critical $L$-value is nonzero.

If no automorphic member of the global $L$-packet determined by $\Pi$ is 
globally distinguished, presumably for reasons of an $L$-value, there is 
no further that we need to proceed in this quest in the global $L$-packet
determined by $\Pi$. So we assume that there is a member in the $L$-packet
determined by $\Pi$ which is globally distinguished, which we can then 
assume to be $\Pi$ itself in our further study.
 
\vspace{4mm}

{\bf Question 2:} Suppose $\Pi=\otimes_v \Pi_v$ is an automorphic 
representation of $G(\A_F)$ such that $\Pi$ is globally 
distinguished by $H(\A_F)$. 
 Let $\Pi'=\otimes_v \Pi'_v$ be an automorphic 
representation of $G(\A_F)$ in the same $L$-packet as $\Pi$ 
such that $\Pi'_v$ is locally 
distinguished by $H(F_v)$ at all the places of $F$.  Then is $\Pi'$
globally distinguished? 

\vspace{4mm}

This is the local-global question being talked about in the title of 
this paper, and which being a question about an individual automorphic 
representation,
and not a question about an $L$-packet, is not governed by an $L$-value,
but keeping the parametrization of automorphic representations in mind 
(due to Labesse-Langlands for $\SL_2$, and then Langlands, Kottwitz, and Arthur), 
should be related to certain finite group of connected components of an
appropriate representation (of the Langlands group). However, in the examples
we deal with in this paper, the local-global principle turns out to be true.

The aim of this work is to initiate such a finer study in the global context 
of some low 
rank cases in detail,
by varying the themes already studied. 
In this work we will consider the two questions above for two basic 
cases. These two cases  will be variations on two rather well-studied examples where 
we  change the groups involved slightly allowing us to consider nontrivial 
local and global $L$-packets.

The first example is  one of  $(\GL_2(E),\GL_2(F))$ where $E$ is a quadratic extension of either local or global fields. This came up in the seminal work of Harder, Langlands and Rapoport \cite{harder} which was later pursued by Flicker and Hakim \cite{flicker1,flicker2,hakim}. Global distinction here is characterized by an $L$-function, the Asai $L$-function, having a pole at $s=1$.   We will analyze  questions 1 and 2 for  the related pair 
$(\SL_2(E),\SL_2(F))$. The starting point of this investigation is an 
elementary observation that an automorphic representation of $\GL_2(\A_E)$ 
has a nontrivial period integral on $\SL_2(\A_F)$ if and only if 
it is $\chi$-distinguished for a Gr\"ossencharacter $\chi$ of $\A_F^\times$, 
which is then a condition on the Asai $L$-function twisted by $\chi^{-1}$ to 
have a pole at $s=1$. This allows conclusions about $L$-packets of 
automorphic representations of $\SL_2(\A_E)$, but making conclusions 
about individual automorphic representations of $\SL_2(\A_E)$ is subtler.

The second example that we will consider in this paper is  related to
 the celebrated work of Waldspurger \cite{waldspurger}. 
Here $G = \GL_2(F)$, or more generally the invertible elements of a  quaternion algebra over $F$, and 
$H$ is the  torus defined by a quadratic algebra $E/F$. In this case, $\Pi$ is globally $\chi$-distinguished for a 
Gr\"ossencharacter $\chi: \A_E^\times/E^\times \rightarrow \C^\times$  
if and only if each $\Pi_v$ is locally $\chi_v$-distinguished and $L(\frac{1}{2},{\rm BC}(\Pi) \otimes \chi^{-1})\neq 0$, where ${\rm BC}(\Pi)$ denotes the base change lift of $\Pi$ to $\GL_2({\mathbb A}_E)$. The local picture is well understood 
 by the work of Saito and Tunnell \cite{saito,tunnell}, and 
involves certain local epsilon factors. 
We will analyze questions 1 and 2 above for the related pair $(\SL_2(F),E^1)$ where $E^1$ is the subgroup of $E^\times$ of norm 1 elements. 
It may be noted that there are many non-conjugate embeddings of $E^1$ inside $\SL_2(F)$; we will fix one such embedding; our answers
do not depend on this initial fixing of an embedding of $E^1$ inside $\SL_2(F)$.

In the first example, $(\GL_2(E),\GL_2(F))$, the local-global principle 
almost holds. 
If each $\Pi_v$ is locally distinguished, then $\Pi$ is either 
globally distinguished or is globally distinguished with respect 
to the quadratic character $\omega$ associated to 
$E/F$ \cite{harder}. Thus, if each $\Pi_v$ is distinguished and if at least 
one $\Pi_v$ is not $\omega_v$-distinguished, then $\Pi$ is globally distinguished. In particular, if $\Pi_v$  is  square integrable representation 
at least at one place $v$ of $E$ which is inert over $F$, then $\Pi$ is globally distinguished if and 
only if it is locally distinguished. This follows since a discrete series representation of $\GL_2(E_v)$, once distinguished
by $\GL_2(F_v)$, cannot be $\omega_v$-distinguished.

In \cite{anand2} we had constructed 
an  example of an automorphic representation $\Pi$ on $\SL_2(\A_E)$ 
where each $\Pi_v$ is a locally distinguished representation of $\SL_2(E_v)$ 
but no member of the $L$-packet of $\Pi$ is globally  distinguished.
In this paper, we give a positive answer to question 1 in some situations, but 
have  not succeeded  in getting a complete understanding of it.

\begin{theorem}\label{1}
Let $\Pi$ be a cuspidal representation of $\SL_2({\mathbb A}_E)$. If $\Pi$ appears in the restriction of a CM representation of $\GL_2({\mathbb A}_E)$, 
assume that there is at least one square integrable component at a
place of $E$ which is inert over the corresponding place $v_0$ of $F$. In the CM case, 
assume that either $\Pi$ is CM by 
three distinct quadratic extensions of $E$, or if it is CM by a unique quadratic extension of $E$, 
then at the place 
$v_0$, the local component is also CM by a unique quadratic extension of $E_{v_0}$ 
(or more generally, it is CM only by quadratic extensions which are Galois 
over $F_{v_0}$).
Suppose each $\Pi_v$ is distinguished with respect to $\SL_2(F_v)$. Then 
there is a cuspidal representation in the $L$-packet of $\Pi$ which is distinguished with respect to $\SL_2({\mathbb A}_F)$.   
\end{theorem}

The question 2 has a complete answer in the following theorem.

\begin{theorem}\label{2}
Let $\Pi$ be a cuspidal representation of $\SL_2({\mathbb A}_E)$ which is 
globally distinguished by $\SL_2({\mathbb A}_F)$. 
 Let $\Pi'=\otimes_v \Pi'_v$ be an automorphic 
representation of $\SL_2({\mathbb A}_E)$ in the same $L$-packet as $\Pi$ 
such that $\Pi'_v$ is locally 
distinguished by $\SL_2(F_v)$ at all the places of $F$.  Then $\Pi'$
is globally distinguished. 
\end{theorem}

A key ingredient in the proof of Theorem 1.1 is the multiplicity one theorem
for automorphic representations of $\SL_2(\A_F)$ due to \cite{dinakar}, whereas Theorem 1.2 is proved via an 
exact determination of the fibers of the 
Asai lift from automorphic representations on $\GL_2(\A_E)$
to automorphic representations on $\GL_4(\A_F)$,
completing an earlier work of Krishnamurthy in \cite{muthu}.

In the example considered by Waldspurger, $(\GL_2(F),E^\times)$, 
unlike in the first one, there is a 
genuine global obstruction to global distinction, and this is the 
vanishing of the central value of the base change $L$-function.    In our 
study of the pair $(\SL_2(F),E^1)$, we are naturally led into some questions about $L$-functions.

According to a well-known result of Friedberg and Hoffstein \cite{hoffstein}, for an automorphic 
representation $\pi$ on $\GL_2(\A_F)$, there are 
infinitely many quadratic characters $\eta$, with prescribed local behavior 
at finitely many places, 
such that the twisted  $L$-values, $L(\frac{1}{2},\pi \otimes \eta)$, are nonzero, provided the global root number $\epsilon(\frac{1}{2},\pi)$ 
is one 
possibly after twisting $\pi$  by a quadratic character (see also \cite{jacquet1}). This latter 
condition on the global root number of $\pi$ is automatic if $\pi$  has at least one 
square integrable component \cite{waldspurger1}. 
For the analysis of the special linear analogue of the second example, 
one needs to understand (a special case) of the following simultaneous nonvanishing problem, stated
as a conjecture.

\begin{conj}\label{conj}
 Let $\Pi_1$ and $\Pi_2$ be two cuspidal representations of $\GL_2({\mathbb A}_F)$. Let $\eta$ be a quadratic character  such that \[\epsilon(\frac{1}{2},\Pi_i\otimes \eta)=1\]
for those $\Pi_i$ which are self-dual among $\{\Pi_1,\Pi_2\}$.
Then there are infinitely many quadratic characters $\eta^\prime$, which agree with $\eta$ at any finitely many prescribed places of $F$, such that
\[L(\frac{1}{2},\Pi_1\otimes \eta^\prime)\neq 0 \neq L(\frac{1}{2},\Pi_2 \otimes \eta^\prime).\] 
\end{conj}   

Assuming the conjecture, we give a positive answer to question 1 in this 
case, once again assuming that a local component is discrete series.
In fact, this  paper 
emphasizes the role that a discrete series
local component of an automorphic representation 
might make to a global result: a local condition with 
a global effect,  and from the example in \cite{anand2} we know that the global result 
fails without some local conditions.

\begin{theorem}\label{3} Let $D$ be a quaternion  algebra over a number field $F$, with $E$ a quadratic subfield of $D$.
Let $\Pi = \otimes_v \Pi_v$ be a cuspidal representation of $\SL_1(D)({\mathbb A}_F)$ 
with at least one square integrable component at a place $v_0$ of $F$ which we assume is of odd residue characteristic if $E$ is inert 
and $D$ split at $v_0$.  If each $\Pi_v$ is distinguished with respect to 
$E_v^1$, then there is a cuspidal representation in the $L$-packet of $\Pi$ 
which is distinguished with respect to ${\mathbb A}_E^1$.
\end{theorem}

We have also achieved a positive answer to question 2 assuming Conjecture \ref{conj}, 
but only in the case when the global $L$-packet associated to the automorphic 
representation $\Pi$ is finite. In the more general case, we need a finer
version of Conjecture \ref{conj}, for which we refer the reader to section 10.

We end the introduction by noting the role played by analytic number theory
(simultaneous non-vanishing of central $L$-values in this case) in questions on 
automorphic representations; whether one implies the other, or the other 
way around, only  time will tell.

\vspace{4mm}

\noindent{\bf Acknowledgements:} We were inspired to consider this work by  
a question of  Vinayak Vatsal about the $\SL(2)$ analogue of 
Waldspurger's theorem, in which he also suggested 
that since the $L$-function that appears in Waldspurger's theorem 
does not make sense
for $\SL(2)$, there should be no $L$-function condition for  
the non-vanishing of toric
period integrals for $\SL(2)$! 
We have shown here  that although this would be a consequence
of a `standard conjecture' in analytic number theory, 
we have not managed to prove an
unconditional theorem except in the case of split torus. We thank 
Vatsal for the initial impetus to this work.

\section{Period integral for $\GL_2$ versus $\SL_2$}

Suppose that $\widetilde{\pi}$ is a cuspidal automorphic representation of $\GL_2({\mathbb A}_E)$ where $E$ 
is a quadratic extension of a number field $F$. In this section,
we write down an integral formula 
relating the period integral of automorphic
functions in $\widetilde{\pi}$ along $\SL_2(\A_F)$ versus similar period integral on $\GL_2(\A_F)$; this was treated 
in Section 3 of \cite{anand2}. It allows
one to prove that  distinction  by $\SL_2(\A_F)$ 
of an automorphic representation of $\GL_2({\mathbb A}_E)$
with trivial 
central character restricted to $\A_F^\times$  is the same as being 
$\omega$-distinguished for a quadratic character $\omega: \A^\times_F/F^\times \rightarrow \C^\times$, cf. Proposition 3.3 of
\cite{anand2}.  We note as has  been observed in section 3 of \cite{anand2} that an automorphic representation of $\GL_2({\mathbb A}_E)$
with non-trivial period integral on $\SL_2(\A_F)$ has a twist whose central character restricted to $\A_F^\times$  is trivial.

The following is Proposition 3.2 of \cite{anand2}, and is a simple consequence of elementary Fourier analysis.

\begin{prop}\label{global1}
Let $E$ be a quadratic extension of a number field $F$. Let $\phi$ be a cusp 
form on $\GL_2({\Bbb A}_E)$ with central character which is trivial when restricted to $\A_F^\times$. Then
\[\int_{\SL_2(F) \backslash \SL_2(\A_F)}\phi(g)dg = 
\sum_\omega \int_{\GL_2(F) \A_F^\times \backslash \GL_2(\A_F)} \phi(g)\omega(\det g)dg\]
where the sum on the right hand side of the equality sign is over all 
characters $\omega: F^\times \backslash \A_F^\times \rightarrow \Z/2$.
\end{prop}

The following  proposition relates period integrals over 
${\mathbb A}_E^1$
of  automorphic
forms of $\GL_2({\mathbb A}_F)$ with period integrals over ${\mathbb A}_E^\times$. 
We omit the simple proof  based on elementary Fourier analysis.

\begin{prop}\label{global1}
Let $E$ be a quadratic extension of a number field $F$. Let $\phi$ be a cusp 
form on $\GL_2({\Bbb A}_F)$ with trivial central character. Then
\[\int_{E^1\backslash {\Bbb A}_E^1}\phi(g)dg = 
\sum_\eta \int_{E^{\times}{\mathbb A}_F^{\times}\backslash {\Bbb A}_E^{\times}}\phi(g)\eta(g)dg\]
where the sum on the right hand side of the equality sign is over all 
characters $\eta$ of the compact abelian group $E^{\times}{\mathbb A}_F^{\times} \A_E^1\backslash {\Bbb A}_E^{\times} = 
E^{\times}{\mathbb A}_F^{\times}{\mathbb A}_E^{\times 2}\backslash {\Bbb A}^{\times}_E$.
\end{prop}

As a consequence,  we have the following.

\begin{prop}\label{global2}
If $\widetilde{\pi}$ is a cusp form on $\GL_2({\Bbb A}_F)$ with trivial central character 
which is  distinguished by 
${\Bbb A}_E^1$, then there is a Gr\"ossencharacter $\eta$ of $E^{\times}\backslash
{\Bbb A}^{\times}_E$ such that $\widetilde{\pi}$ is $\eta$-distinguished 
for ${\rm GL}_1({\Bbb A}_E)$.
Conversely if $\widetilde{\pi}$ is $\eta$-distinguished for 
some Gr\"ossencharacter $\eta$
of $E^{\times}\backslash {\Bbb A}^{\times}_E$, then $\widetilde{\pi}$ is 
${\Bbb A}_E^1$-distinguished. Hence there is a member of the $L$-packet 
of automorphic representations of 
$\SL_2({\Bbb A}_F)$ determined by $\widetilde{\pi}$ which is globally
${\Bbb A}_E^1$-distinguished.
\end{prop}

\section{Distinction as a functorial lift}

In this section we recast the well-known criterion about distinction of $\GL_2(E)$ representations to $\SL_2$ according to which 
a representation of $\GL_2(E)$ is distinguished or $\omega_{E/F}$-distinguished by $\GL_2(F)$ if and only if 
$\omega_\pi|_{_{F^\times}}=1$ and $\pi^\sigma \cong \pi ^\vee$.

\begin{theorem} Let $E/F$ be a quadratic extension of non-archimedean local fields. Then, an irreducible admissible representation $\pi$ of 
 $\GL_2(E)$ is distinguished by $\SL_2(F)$ if and only if it belongs to the twisted basechange map, i.e., 
a character twist of $\pi$ is a basechange from $\GL_2(F)$.  Exactly the same conclusion about global distinction 
of automorphic representations of $\GL_2(\A_K)$ with respect to $\SL_2(\A_k)$ when $K/k$ is a quadratic extension of
number fields.
\end{theorem}

\begin{proof} We will write the argument below assuming $E/F$ is a quadratic extension of local fields, but the same argument
works verbatim for number fields.

Let TBC  denote the base change map from irreducible admissible representations of $\GL_2(F)$, considered up to twists by characters, to irreducible admissible representations of $\GL_2(E)$, considered up to twists by characters. Thus, any representation in the image of TBC is of the form BC$(\pi^\prime)\otimes \chi$ for a representation $\pi^\prime$ of $\GL_2(F)$ and a character $\chi$ of $E^\times$. 

We claim that the representation $\pi$ of  $\GL_2(E)$ is distinguished by $\GL_2(F)$ with respect to a character $\eta$ of $F^\times$ if and only if $\pi$ 
 is in the image of the twisted base change map. 
Since we are looking at representations modulo character twists, we can assume that $\eta=1$, thus $\pi$ itself is 
distinguished, and therefore by standard results, cf. \cite{flicker2}, it follows that:
\[\omega_\pi|_{_{F^\times}}=1 ~~ \& ~~ \pi^\vee \cong \pi^\sigma.\]
If we write $\omega_\pi=\mu^{-1}\mu^\sigma$ for a character $\mu$ of $E^\times$, then $\pi^\vee \cong \pi^\sigma$ implies  
 that $\pi\otimes \mu$ is Galois stable and hence $\pi$ is in the image of TBC.

Conversely, if $\pi$ is of the form BC$(\pi^\prime) \otimes \chi$, then we prove that $\pi$ is $\SL_2(F)$-distinguished, for which we may as 
well assume that $\pi = {\rm BC}(\pi^\prime)$. 

Let $\omega^\prime$ be the central character of $\pi^\prime$, and let $\grave{}{\omega}$ 
be an extension of $\omega^\prime$ to $E^\times$. Then $\omega_\pi = \grave{}{\omega} \cdot 
\grave{} {\omega}^{\sigma}$, from which it can be checked that 
the representation $\grave{}\pi= \pi \otimes \grave{}{\omega}^{-1} $ has the property
$ \grave{}\pi^\vee \cong \grave{}\pi^{ \sigma} $ and that the  central character of 
$~\grave{} \pi$ restricted to $F^\times$ is trivial. 
This shows that $\grave{} \pi$ is either distinguished or $\omega_{_{E/F}}$ distinguished 
by $\GL_2(F)$ by Theorem 7 of \cite{flicker2}, hence $\pi$ is 
$\eta$-distinguished for
some character $\eta$ of $F^\times$.
\end{proof}

This theorem allows to interpret distinction as a lifting of maps,  
$$
\xymatrix{
& \PGL_2(\C) \times W_k,
\ar@{_(->}[d]\\
W'_k \ar@{-->}[ur]\ar[r] & (\PGL_2(\C) \times \PGL_2(\C)) \rtimes {\rm Gal}{(K/k)},
}
$$
where $k$ can be a local or global field, and in the latter case, $W'_k$ needs to be replaced by the conjectural Langlands group whose irreducible 
$n$-dimensional complex representations classify  cuspidal automorphic representations of $\GL_n(\A_k)$.

In this language, the lifting question in the {\it untwisted} diagram, 
$$
\xymatrix{
& \PGL_2(\C) \times W_k,
\ar@{_(->}[d]\\
W'_k \ar@{-->}[ur]\ar[r] & (\PGL_2(\C) \times \PGL_2(\C)) \times {\rm Gal}{(K/k)},
}
$$
locally asks if two representations of $\GL_2(k)$ are character twists of each other, and globally if they are
twists of each other by a Gr\"ossencharacter of $\A_k$; thus in this case, a theorem of Dinakar Ramakrishnan \cite{dinakar}
guarantees that local lifts in the above diagram imply a global lift, whereas a theorem, or rather a construction of 
Blasius, proves that existence of local lifts does not guarantee global lift when $\PGL_2(\C)$  is replaced by $\PGL_n(\C)$.

\section{Distinction of some member in an $L$-packet for the pair    $(\SL_2(E),\SL_2(F))$}

In this section, we prove Theorem 1.1, which we recall here for the convenience of the reader.

\begin{theorem}\label{1}
Let $\Pi$ be a cuspidal representation of $\SL_2({\mathbb A}_E)$. If $\Pi$ appears in the restriction of a CM representation of $\GL_2({\mathbb A}_E)$, 
assume that there is at least one square integrable component at a
place of $E$ which is inert over the corresponding place $v_0$ of $F$. In the CM case, 
assume that either $\Pi$ is CM by 
three distinct quadratic extensions of $E$, or if it is CM by a unique quadratic extension of $E$, 
then at the place 
$v_0$, the local component is also CM by a unique quadratic extension of $E_{v_0}$ 
(or more generally, it is CM only by quadratic extensions which are Galois 
over $F_{v_0}$).
Suppose each $\Pi_v$ is distinguished with respect to $\SL_2(F_v)$. Then 
there is a cuspidal representation in the $L$-packet of $\Pi$ which is distinguished with respect to $\SL_2({\mathbb A}_F)$.   
\end{theorem}

\begin{proof} Let $\Pi = \otimes_v \Pi_v$ be a cuspidal representation of 
$\SL_2({\mathbb A}_E)$ with each $\Pi_v$ distinguished by $\SL_2(F_v)$. 
Let $\widetilde{\Pi} = \otimes_v \widetilde{\Pi}_v$ be a cuspidal representation of $\GL_2({\mathbb A}_E)$ 
containing $\Pi$.  We claim that there is a Gr\"ossencharacter $\chi$ of ${\mathbb A}_E^{\times}$ such that 
\[\widetilde{\Pi}^\sigma \cong \widetilde{\Pi}^\vee \otimes \chi.\] 
To this end, observe that since for each $v$, 
the representation $\widetilde{\Pi}_v$ of $\GL_2(E_v)$ 
is given to be $\SL_2(F_v)$-distinguished,
there is a character $\eta_v$ of $F_v^{\times}$ such that $\widetilde{\Pi}_v$ is $\eta_v^{-1}$-distinguished with respect to $\GL_2(F_v)$. If $\widetilde{\eta}_v$ denotes an extension of $\eta_v$ to $E_v^{\times}$, then $\widetilde{\Pi}_v \otimes \widetilde{\eta}_v$ is distinguished with respect to $\GL_2(F_v)$, and this implies that \cite{flicker2,hakim}

\[(\widetilde{\Pi}_v \otimes \widetilde{\eta}_v)^\vee \cong (\widetilde{\Pi}_v \otimes \widetilde{\eta}_v)^\sigma,\]
or
\[\widetilde{\Pi}_v^\sigma \cong \widetilde{\Pi}_v^\vee \otimes \eta_v \circ \N m.\]

By a theorem of Ramakrishnan \cite{dinakar}, if two automorphic representations of $\GL_2({\mathbb A}_E)$ 
are locally twists of each other at all places of a number field $E$, then they are
globally twists of each other by a Gr\"ossencharacter $\chi$ on ${\mathbb A}_E^\times$, proving our claim
that $\widetilde{\Pi}^\sigma \cong \widetilde{\Pi}^\vee \otimes \chi$.

For the proof of Theorem \ref{1}, it suffices to prove that there is a Gr\"ossencharacter $\chi$ on ${\mathbb A}_E^\times$ with $\chi^\sigma = \chi$,
and with $\widetilde{\Pi}^\sigma \cong \widetilde{\Pi}^\vee \otimes \chi$, 
because then one can write $\chi^{-1} = \mu \mu^\sigma$, which means that 
$(\widetilde{\Pi} \otimes \mu
)^\sigma \cong (\widetilde{\Pi} \otimes \mu)^\vee$, 
and hence $\widetilde{\Pi} \otimes \mu$ is either $\GL_2({\mathbb A}_F)$ distinguished, or $\omega_{E/F}$ distinguished by $\GL_2({\mathbb A}_F)$. 
This means that some member in the global 
$L$-packet determined by the automorphic 
representation $\Pi$ of $\SL_2({\mathbb A}_E)$ has nontrivial period
integral on $\SL_2({\mathbb A}_F)$.

We first dispose off the case when $\widetilde{\Pi}$ is non-CM, i.e., it is not automorphically induced from a Gr\"ossencharacter of a quadratic extension, equivalently,  there is no nontrivial self-twist by a Gr\"ossencharacter. In this case,
$\widetilde{\Pi}^\sigma \cong \widetilde{\Pi}^\vee \otimes \chi$ implies 
$\widetilde{\Pi}^\sigma \cong \widetilde{\Pi}^\vee \otimes \chi^\sigma$, and therefore
since $\widetilde{\Pi}$ does not have CM, $\chi^\sigma = \chi$.

We now assume that $\widetilde{\Pi}$ has CM and that there is a place of $F$  inert in $E$, say $v_0$ in $E$, such that $\widetilde{\Pi}_{v_0}$ is square integrable. Suppose the assertion of the theorem is not true. Then 
$\widetilde{\Pi}^\sigma \cong \widetilde{\Pi}^\vee \otimes \chi$ with $\chi \neq \chi^\sigma$.

By the reciprocity isomorphism $W_F^{ab} \cong {\mathbb A}_F^\times/F^\times$, where $W_F$
is the Weil group of $F$, 
 operations on $W_F$ and its subgroups  have their analogues on the automorphic side. In particular, by $\boxtimes$ and $\boxplus$, we denote the operations corresponding to tensor product and direct sum. Our assumption $\widetilde{\Pi}^\sigma \cong \widetilde{\Pi}^\vee \otimes \chi$
 can be rewritten as
\[\widetilde{\Pi} \boxtimes \widetilde{\Pi}^\sigma 
= \chi \boxplus \chi^\sigma \boxplus {\rm BC}(\rho),\]
for a certain two dimensional representation $\rho$ of $W_F$. 

If $\widetilde{\Pi}$ 
has CM by three quadratic extensions, then $\widetilde{\Pi}$ has self-twists by three quadratic characters,
forcing $\widetilde{\Pi} \boxtimes \widetilde{\Pi}^\sigma $ which contains a character to be a sum of four characters permuted amongst themselves by $\sigma$. Therefore,
${\rm BC}(\rho)$ is  a sum of two characters which we assume is of the form $\mu  \boxplus \mu^\sigma$ with $\mu  \not =  \mu^\sigma$, as  the other possibilities create  a $\sigma$-invariant character of $\A_E^\times/E^\times$. We will now look at the above decomposition at the place $v_0$ of $E$:
\[\widetilde{\Pi}_{v_0} 
\boxtimes \widetilde{\Pi}_{v_0}^\sigma 
= \chi_{v_0} \boxplus \chi_{v_0}^\sigma \boxplus {\rm BC}(\rho_{v_0}).\]
Note that $\chi_{v_0} \neq \chi_{v_0}^\sigma$, 
since $\widetilde{\Pi}_{v_0}$ is square integrable 
and thus corresponds to an irreducible representation of $W_{E_{v_0}}$, and therefore 
each character in the decomposition of $\widetilde{\Pi}_{v_0} \boxtimes \widetilde{\Pi}_{v_0}^\sigma $
appears with multiplicity 1 by Schur's lemma, 
forcing $\chi_{v_0} \not = \chi^\sigma_{v_0}$.  Further in the case,
${\rm BC}(\rho) = \mu \boxplus  \mu^\sigma$, $\mu \not = \mu^\sigma$, we again have $\mu_{v_0} \not = 
\mu^\sigma_{v_0}$, which is contradictory to our assumption of having a $\sigma$-invariant
character inside $\widetilde{\Pi}_{v} \boxtimes \widetilde{\Pi}_{v}^\sigma $ at  all places $v$ of 
$E$.

We next deal with the other case listed in Theorem \ref{1}:  when $\widetilde{\Pi}$ has CM by a unique quadratic
extension of $E$, and $\widetilde{\Pi}_{v_0}$ also has CM by a unique quadratic extension 
of $E_{v_0}$, call it $M_0$. Let $\omega_0$ be the 
corresponding 
 quadratic character of $E^\times_{v_0}$ associated to $M_0$. By
the identity $\widetilde{\Pi}_{v_0}^\sigma \cong \widetilde{\Pi}_{v_0}^\vee \otimes \chi_{v_0}$,
we find that the extension $M_0$ of $E_{v_0}$ is $\sigma$-invariant, hence the character $\omega_0$ 
of $E_{v_0}$ is $\sigma$-invariant. The  identities 
$$\widetilde{\Pi}_{v_0}^\sigma \cong \widetilde{\Pi}_{v_0}^\vee 
\otimes \chi_{v_0} \cong 
\widetilde{\Pi}_{v_0}^\vee \otimes \eta_{v_0} \circ \N m
,$$ 
imply that $\chi_{v_0} = \eta_{v_0}\circ \N m,$ or $\chi_{v_0} = \omega_0 \eta_{v_0} \circ \N m.$  Since 
$\omega_0$   is $\sigma$-invariant, in either case $\chi_{v_0}$ is $\sigma$-invariant,
which is a contradiction to Schur's lemma as already observed before.

Clearly, the same proof works when $\widetilde{\Pi}_{v_0}$  has CM by three 
quadratic extension of $E_{v_0}$ which are all Galois over $F_{v_0}$ since all
that mattered for the proof was that the corresponding quadratic characters of
$E^\times_{v_0}$ are $\sigma$-invariant.  
\end{proof}
\section{Tensor Induction,  or Asai Lift}

In the study of automorphic representations of $\GL_2(\A_E)$ which are distinguished by
$\GL_2(\A_F)$, Asai lift plays an important role, and it does so in our work on the
corresponding questions for $\SL_2$.
The specific aim of this section will be to determine the fibers of the Asai lift
$\widetilde{\pi} \rightarrow {\rm As}(\widetilde{\pi})$ 
from automorphic representations of $\GL_2(\A_E)$
to automorphic representations of $\GL_4(\A_F)$. 
This question was discussed by Muthu Krishnamurthy in \cite{muthu}; however, in the 
case 
where it really concerns us, the case of CM representations of $\GL_2$, his 
result was incomplete exactly in the place where it matters to us. We have 
completed his work in this section.\footnote{There is a recent preprint of 
Muthu Krishnamurthy, cf. \cite{muthu2} where he too completes his earlier work.}

We begin this section by carefully 
recalling the notion of {\it tensor induction}, also 
called Asai lift in a particular case (when the subgroup involved is of 
index 2), which is a purely group theoretic notion.

Let $H$ be a subgroup of a group $G$ of finite index $n$, and ${\mathcal G}$ 
an arbitrary group. Define ${\mathcal G}^{G/H}$ 
to be the set of all
set theoretic maps $\phi: G \rightarrow {\mathcal G}$ such that $\phi(gh) = \phi(g)$ for all $g \in G, h \in H$. Clearly ${\mathcal G}^{G/H}$ is a group 
with a natural action of $G$ on the left, so we can form 
the semi-direct product ${\mathcal G}^{G/H} \rtimes G$.

It is easy to prove the following lemma, which is nothing but a form of 
Frobenius reciprocity for induced representations in this context.

\begin{lem} There exists a natural bijection 
$$\Hom(H,{\mathcal G})/ {\sim}  \longleftrightarrow \Hom(G, {\mathcal G}^{G/H} \rtimes G)/{\sim},$$
where we consider only those homomorphisms in $\Hom(G, {\mathcal G}^{G/H} 
\rtimes G),$ whose composition with the natural map from 
${\mathcal G}^{G/H} \rtimes G$ to $G$ is the identity map from $G$ to $G$; the equivalence relation on the left hand side
is conjugation by ${\mathcal G}$, and on the right is conjugation by ${\mathcal G}^{G/H}$. 
\end{lem}

Now given a representation $(\pi,V)$ of ${\mathcal G}$, it gives rise
to a representation $\otimes^{G/H}V$ of ${\mathcal G}^{G/H}$ which clearly 
extends to one of the semi-direct product ${\mathcal G}^{G/H} \rtimes G$.
Taking ${\mathcal G}$ to be $\GL(V)$ with its natural representation on $V$,
the previous lemma allows one to associate to a representation
$(\pi,V)$ of $H$ of dimension $d$, a representation of $G$, to be denoted
by ${\rm As}(V)$, of dimension $d^n$, called the tensor induction, or the Asai lift 
of the representation $(\pi,V)$ of $H$.

For a vector space  $W$ over $\C$ equipped with a quadratic form $q$ on it, 
there is the notion of the orthogonal similitude group $\GO(W)$, defined by 
$$\GO(W)= \{ g \in \GL(W)| q(gw) = \lambda(g) q(w)\,\,\,\,  \forall w \in W\};$$
the map $g \rightarrow \lambda(g) \in \C^\times$ is a character on $\GO(W)$, called the similitude character. If $W$ is of 
even dimension, the special orthogonal similitude group, denoted by $\GSO(W)$,
which is the connected component of identity of $\GO(W)$, is defined by
$$\GSO(W) = \{g \in \GO(W)| \lambda(g)^{\dim W/2} = \det g\}.$$

The following well-known result lies at the basis of our proof. 
It can itself be considered as a local-global principle for 
orthogonal groups, 
eventually responsible for multiplicity one (conjecture) for 
automorphic representations of orthogonal groups, or more generally any 
classical group, cf. \cite{larsen}.

\begin{lem} Let $W$ be a finite dimensional vector space over $\C$ 
together with a quadratic form $q$ on it. Suppose $\pi_1$ and $\pi_2$ 
are two representations of a group $G$ into $\GO(W)$ such that the similitude
characters $\lambda_1$ and $\lambda_2$ of $\pi_1$ and $\pi_2$ are the same. 
Then the representations $\pi_1$ and $\pi_2$ of $G$ with values in $\GO(W)$ 
are equivalent, i.e., conjugate  in $\GO(W)$, if and only if they are
equivalent in $\GL(W)$.
\end{lem}
With these generalities in place, we now come to the special situation
afforded by 2 dimensional representations of a subgroup $N$ of index 2 in a 
group $G$. In this case, we find it more convenient  to use a concrete
realization of $\GO(4,\C)$, which we realize on the space $M(2,\C)$ of $2 \times 2$ matrices 
with $X \rightarrow \det X$ as the quadratic form. Clearly, $(A,B) \in \GL_2(\C) \times \GL_2(\C)$
acting on $M(2,\C)$ as $X \rightarrow A \cdot X \cdot {}^tB$ defines a mapping from $\GL_2(\C) \times \GL_2(\C)$
onto $\GSO(4,\C)$, and the involution  $X \rightarrow {}^tX$ belongs to ${\rm O}(4,\C)$ but not to $\SO(4,\C)$.
Thus, we have an exact sequence,

$$1 \rightarrow \C^\times \rightarrow [\GL_2(\C) \times \GL_2(\C)]\rtimes {\Z}/2
\rightarrow \GO(4,\C)\rightarrow 1,$$
where $\C^\times$  sits inside $\GL_2(\C) \times \GL_2(\C)$ as scalar matrices $(z,z^{-1})$.

As noted earlier, a representation $\pi_1$ of $N$ into $\GL_2(\C)$ gives rise to a 
homomorphism of $G$ into $[\GL_2(\C) \times \GL_2(\C)]\rtimes {\Z}/2$
whose projection to ${\Z}/2$ is nothing but the natural projection from $G$
to $G/N = \Z/2$. It will be convenient at this point to use the language
of cohomology of groups (with non-abelian coefficients). In this language,
we have an exact sequence of $G$-groups:
$$1 \rightarrow \C^\times \rightarrow \GL_2(\C) \times \GL_2(\C)
\rightarrow \GSO(4,\C)\rightarrow 1,$$
where $\C^\times$ is the $G$-module on which $N$ operates trivially, and
the nontrivial element of $G/N$ operates on $\C^\times$ by $z\rightarrow z^{-1}$. 
This exact sequence of $G$-modules gives rise to an exact sequence of 
pointed sets, which because $\C^\times$ is a central subgroup of
$\GL_2(\C) \times \GL_2(\C)$ 
(sitting as scalar matrices $(z,z^{-1})$),
is in fact an exact sequence of sets as can be easily seen (this is Proposition 42 in 
part I, $\S  5$ of Serre's book `Galois Cohomology'):
$$ H^1(G,\C^\times) \rightarrow H^1(G,\GL_2(\C) \times \GL_2(\C))
\rightarrow H^1(G,\GSO(4,\C)).$$
In terms of group cohomology, we have the identification,
$$\frac{\Hom[G,[\GL_2(\C) \times \GL_2(\C)]\rtimes {\Z}/2] }{\sim}
\longleftrightarrow H^1(G, \GL_2(\C) \times \GL_2(\C)),$$
where $\sim$ denotes the equivalence relation on the set of
homomorphisms given by conjugation by $\GL_2(\C) \times \GL_2(\C)$.

It follows that two homomorphisms $\phi_1$ and $\phi_2$ 
of $G$ to $(\GL_2(\C) \times \GL_2(\C))\rtimes {\Z}/2$ which give
rise to the same representation of $G$ with values in $\GO(4, \C)$
differ by an element of $H^1(G,\C^\times)$ which we calculate in 
the following lemma.

\begin{lem} Let $N$ be an index 2 subgroup of a group $G$. Let $\C^\times$ be the $G$ module on which $N$ operates trivially, and the non-trivial
element of $G/N$ acts on $\C^\times$ by $z\rightarrow z^{-1}$. Then $H^1(G,\C^\times)$ can be identified to those characters of $N$ with values
in $\C^\times$ whose transfer to $G$ is trivial.
\end{lem}
\begin{proof}
Note the exact sequence,
$$ \begin{CD}0 @>>>  H^1(G/N,\C^\times)  @>>> H^1(G,\C^\times)  @>>> H^1(N,\C^\times)^{G/N}  @>>>  H^2(G/N,\C^\times) , \end{CD}$$
in which $G/N = \Z/2$.
From well-known calculations on cohomology of cyclic groups, it is easy to see that $H^1(\Z/2, \C^\times) = 0$, and $H^2(\Z/2, \C^\times) = \Z/2$. (In 
this lemma $\C^\times$ comes equipped with the action of $G/N=\Z/2$ by $z \rightarrow z^{-1}$.)
So, the above exact sequence can be written as:
$$ \begin{CD}0 @>>> H^1(G,\C^\times)  @>>> H^1(N,\C^\times)^{G/N}  @>>>  H^2(G/N,\C^\times) . \end{CD}$$
Since $N$ operates trivially on $\C^\times$, $H^1(N,\C^\times)$ is simply the character group of $N$. The group $G/N$ operates on $H^1(N,\C^\times)$
by sending a character $\phi \in H^1(N,\C^\times)$ to  the character $\phi^g(n) = g^{-1} \phi(gng^{-1})$ of $N$. It follows that 
$H^1(N,\C^\times)^{G/N}$ can be identified to  the group of  characters of $N$ for which $\phi^{-1}(n) = \phi(g_0ng^{-1}_0)$ where $g_0$ is any element of
$G$ not in $N$; these are simply the characters of $N$ which when composed with  the transfer map from 
$G/[G,G]$ to $ N/[N,N]$ are trivial on $N$. To get the conclusion of the lemma, we need to prove that among these characters of
$N$, those which go to 0 under the boundary map : $\begin{CD}H^1(N,\C^\times)^{G/N}  @>>>  H^2(G/N,\C^\times) ,\end{CD}$ are exactly those whose
transfer to $G$ is trivial (and not just restriction to the subgroup $N$ which is of index 2). Observe that the transfer map from $G$ to $N$ 
on elements of $G$ outside $N$ is simply the squaring map $g\rightarrow g^2$. So we need to prove that if a character $\phi$ 
in $H^1(N,\C^\times)^{G/N}$ goes to zero in $H^2(G/N,\C^\times)$, then $\phi(g_0^2)=1$ where $g_0$ is any element of $G$ not in $N$.
For this we need to interpret this boundary map, which is nothing but the push-out diagram
under the homomorphism $\phi: N \rightarrow \C^\times$ of the exact sequence: 
$0 \rightarrow  N \rightarrow G \rightarrow \Z/2 \rightarrow 0 $, thus fits in the following commutative diagram:

 $$
\begin{CD}
0 @>>>  N @>>> G  @>>> \Z/2 @>>> 0\\ 
& & @V  VV @VVV @V  VV \\
0 @>>> \C^\times  @>>> E  @>>> \Z/2 @>>> 0.
\end{CD}
$$

\vspace{4mm}

To say that the push-out diagram is trivial, i.e. the short exact sequence $$\begin{CD}0 @>>> \C^\times  @>>> E  @>>> \Z/2 @>>> 0 \end{CD}$$ splits, 
 is clearly equivalent to say that 
$\phi(g_0)^2=1$, so the proof of the lemma is completed.\end{proof}

In the following proposition, for a character
$\chi$ of $N$, let $r(\chi)$ be the character of 
 $G$ obtained from $\chi$ using the transfer map from
$G/[G,G]$ to $N/[N,N]$. (Note that $r(\chi)$ is the special case
of the tensor induction corresponding to one dimensional representations.) The previous lemma proves the following proposition which is the main result of
this section.

\begin{prop}\label{prop5.4}
Let $N$ be an index 2 subgroup of a group $G$, and let $\pi_1$ and $\pi_2$ 
be 2  two dimensional representations of $N$, with ${\rm As}(\pi_1)$ and ${\rm As}(\pi_2)$ 
of dimension 4 their tensor induction to $G$. Assume that 
$r(\det \pi_1) = r(\det \pi_2)$. Then ${\rm As}(\pi_1) \cong 
{\rm As}(\pi_2)$ if and only if $\pi_1 \cong \pi_2 \otimes \chi$, or  
$\pi_1^\sigma \cong \pi_2 \otimes \chi$, for a character
$\chi$ of $N$ with $r(\chi)=1$.
\end{prop}

\noindent{\bf Remark:} The abstract 
group theoretic proof given above for the fibers 
of the map $\pi \rightarrow {\rm As}(\pi)$, yields an exact description of the 
fibers of the Asai lift  from automorphic forms on 
$GL_2(\A_E)$ to automorphic forms on $\GL_4(\A_F)$  for CM 
representations of $\GL_2(\A_E)$. Luckily, non-CM 
representations were already handled by Krishnamurthy in \cite{muthu}, so this
description of the fibers holds in all cases.

\vspace{4mm}

The proof of the previous  proposition 
also gives a  proof of the following proposition 
which however we will not have occasion to use.

\begin{prop}\label{prop5.5} Let $V_1,V_2,W_1,W_2$ be two dimensional representations 
of a group $G$ such that 
$$V_1 \otimes V_2 \cong W_1 \otimes W_2,$$
and $$\det(V_1) \det(V_2) = \det (W_1) \det (W_2),$$
then there exists a character $\chi$ of $G$ such that
$$V_1 = \chi \otimes W_1, \,\,\, V_2 = \chi^{-1}\otimes W_2,$$
or,
$$V_2 = \chi \otimes W_1, \,\,\, V_1 = \chi^{-1} \otimes W_2.$$
\end{prop}

\noindent{\bf Remark:} A weaker version of this proposition was proved in \cite{kumar}
in which $V_1$ and $V_2$ were assumed to be non-CM representations, which 
went into the proof of \cite{muthu}.

\vspace{4mm}

\noindent{\bf Question:} Since $(U_1 \otimes U_2) \otimes U_3 \cong U_1 \otimes (U_2 \otimes U_3)$,
there is no simple way to generalize the previous proposition for larger dimensional representations,
except possibly when, in the notation of the proposition, $\dim V_1$ and $\dim V_2$ are prime. Similarly, since ${\rm As}(U_1\otimes U_2^\sigma) 
\cong {\rm As}(U_1 \otimes U_2)$, there is no simple generalization of the proposition
 about fibers of the
Asai lift of two dimensional representations 
except possibly when dealing with representations of prime dimension. We do not know if in these special
cases in which representations involved are of prime dimension, fibers of Asai lift or of tensor product are as described in propositions \ref{prop5.4} and \ref{prop5.5}.

\section{Local-global principle for the pair $(\SL_2(E),\SL_2(F))$}

In this section, we  work inside an $L$-packet 
to prove  the local-global principle for automorphic 
representations  of $\SL_2(\A_E)$ with respect to $\SL_2(\A_F)$, i.e. 
Theorem 1.2 recalled here for the convenience of the reader. 

\begin{theorem}\label{2}
Let $\Pi$ be a cuspidal representation of $\SL_2({\mathbb A}_E)$ which is 
globally distinguished by $\SL_2({\mathbb A}_F)$. 
 Let $\Pi'=\otimes_v \Pi'_v$ be an automorphic 
representation of $\SL_2({\mathbb A}_E)$ in the same $L$-packet as $\Pi$ 
such that $\Pi'_v$ is locally 
distinguished by $\SL_2(F_v)$ at all the places of $F$.  Then $\Pi'$
is globally distinguished. 
\end{theorem}

Before we begin the proof of this theorem, we make a review of the theory of $L$-packets, both locally
as well as globally for $\SL_2$, which is due to Labesse-Langlands \cite{labesse}, and 
also review some of our own work, cf. \cite{anand1,anand2} 
about distinguished representations, relevant for this study.

We deal with the pair 
$(\SL_2(\A_E),\SL_2(\A_F))$ in this section, making an essential use of the 
theory of Whittaker
models,
and then in a later section observe that some of our work
carries over to the more general situation of  the group of 
norm one elements of a quaternion algebra.

Note that the group $\A_E^\times$ sitting inside $\GL_2(\A_E)$ as 
$$\left (\begin{array}{cc} x & 0 \\ 0 & 1\end{array} \right ),$$
operates on $\SL_2(\A_E)$ via conjugation action, and therefore on the 
set of isomorphism classes of representations of $\SL_2(\A_E)$. 
The orbit 
of $\Pi= \otimes \Pi_v$, an irreducible representation of $\SL_2(\A_E)$,  
under the action of $\A_E^\times$ is precisely the  global 
$L$-packet of representations of $\SL_2(\A_E)$ containing $\Pi$. Let
$G_\Pi \subset \A_E^\times $,  $G_\Pi= \prod G_v$ be the stabilizer of the isomorphism class of the 
representation $\Pi 
=\otimes \Pi_v$, where $G_v$ is the stabilizer inside $E_v^\times$ 
of the isomorphism class of the representation $\Pi_v$ of $\SL_2(E_v)$. 
It can be seen that $G_v$ 
contains ${\mathcal O}_{v}^\times$ for almost all primes $v$ of $E$, where
${\mathcal O}_v$ is the ring of integers of $E_v$, 
and so $G_\Pi$ is an open (and hence closed) 
subgroup of  $\A_E^\times$. 

Clearly, the action of $E^\times$ on $\SL_2(\A_E)$ 
takes  automorphic representations of $\SL_2(\A_E)$ to automorphic representations of $\SL_2(\A_E)$. 
Since every automorphic representation of $\SL_2(\A_E)$ 
must have a Whittaker model for a nontrivial 
character of $\A_E/E$, and any two  nontrivial characters
of $\A_E/E$ are conjugate by $E^\times$, it follows 
from the uniqueness of Whittaker models (for $\GL_2(\A_E)$!) that $E^\times$ acts transitively on the set of automorphic 
representations of $\SL_2(\A_E)$ which are in the same $L$-packet as $\Pi$. 

There is another way of interpreting $G_\Pi = \prod G_v$. For this, 
let $\widetilde{\Pi}$ be an automorphic representation of $\GL_2(\A_E)$ containing $\Pi$. 
Then, for a character
$\omega: \A_E^\times \rightarrow \C^\times$, $\widetilde{\Pi} \otimes \omega \cong \widetilde{\Pi}$ 
if and only if $\omega$ is trivial on $G_\Pi$. This implies  that 
$\A_E^\times/(E^\times G_\Pi)$ is a finite group whose character group
is nothing but the finite group of Gr\"ossencharacters $\omega$ of 
$\A_E^\times/E^\times$ such that $\widetilde{\Pi} \otimes \omega \cong \widetilde{\Pi}$.

From the previous observations, we note that a  representation 
of $\SL_2(\A_E)$ which belongs to the $L$-packet determined by $\Pi$
determines an element of the finite group $\A_E^\times/(E^\times G_\Pi)$ (which 
is known to be either $\{1\}, \Z/2,$ or $\Z/2 \oplus \Z/2$),
which is trivial if and only if the corresponding representation 
is automorphic. This result  due to Labesse and Langlands \cite{labesse}
remains true for division algebras, but this simple proof does not work.

We next review the work in \cite{anand1,anand2} relevant for the local-global study
of the pair $(\SL_2(\A_E),\SL_2(\A_F))$. 

It follows from Theorem 4.2 of \cite{anand2} that if $\Pi$ is globally distinguished 
by $\SL_2(\A_F)$, then $\Pi$ has a Whittaker model with respect to a character
$\psi: \A_E/(E\A_F) \rightarrow \C^\times$.
Conversely, by the same theorem of \cite{anand2} if 
$\Pi$ has Whittaker model with respect to $\psi : \A_E/(E\A_F) \rightarrow \C^\times$, 
and some
member in the $L$-packet determined by $\Pi$  is globally distinguished, then
such a $\Pi$ is itself globally distinguished. 
A similar local result also holds: in a local $L$-packet of $\SL_2(E_v)$ 
where at least one representation is distinguished by $\SL_2(F_v)$, the $\SL_2(F_v)$-distinguished representations are 
precisely those which have a Whittaker model with respect to a nontrivial character of $E_v/F_v$ (this follows from Lemma 3.1 and 
Lemma 3.2 of \cite{anand1}), hence since $F_v^\times$ 
acts transitively on non-trivial characters of $E_v/F_v$, $F_v^\times$ acts transitively 
on distinguished members of an $L$-packet of representations of $\SL_2(E_v)$. 

Define groups, 
\begin{eqnarray*}
H_0&=& \A_E^\times, \\
H_1 & = & \A_F^\times G_\Pi, \\
H_2 & =& E^\times G_{\Pi}, \\
H_3 & = &  F^\times G_{\Pi}.
\end{eqnarray*}

From  these theorems due to Labesse-Langlands \cite{labesse}, and the theorems due to the authors in
\cite{anand1}, \cite{anand2} we deduce the following:

\begin{enumerate}
\item The set $H_0 \cdot \Pi$ is 
the set of $L$-packet of representations of $\SL_2(\A_E)$ determined by $\Pi$.

\item The set $H_1 \cdot \Pi$ is the set of locally distinguished representations
in the $L$-packet of $\SL_2(\A_E)$ determined by $\Pi$.

\item The set $H_2 \cdot \Pi$ is the set of automorphic representations
in the $L$-packet of $\SL_2(\A_E)$ determined by $\Pi$.

\item The set $H_3 \cdot \Pi$ is the set of globally distinguished representations
in the $L$-packet of $\SL_2(\A_E)$ determined by $\Pi$.
\end{enumerate}

Clearly $H_1 \cap H_2$ contains $H_3$, and  $ (H_1 \cap H_2)/ H_3$ 
measures the obstruction to locally distinguished automorphic representations 
to be globally distinguished; equivalently, a locally distinguished
automorphic representation $\Pi$ of $\SL_2(\A_E)$ determines an element 
$h_\Pi$ of $H_1 \cap H_2$ such that $\Pi$ is globally distinguished if and 
only if $h_\Pi \in H_3$.
We will in fact prove that $(H_1\cap H_2)/H_3$ is trivial 
by proving that its character group is trivial. 

Let $X(A)$ denote the character group of a locally compact abelian group $A$.

Noting that 
$(H_1\cap H_2)/H_3$ 
is nothing but the kernel of the map,
$$H_1/H_3 \rightarrow H_0/H_2,$$
the character group of $(H_1\cap H_2)/H_3$ is the cokernel of
 the natural map $$X(H_0/H_2) \rightarrow X(H_1/H_3).$$
We note that the mapping of the character groups is 
simply the map taking
a character $\alpha$ of $H_0$ which is trivial on $H_2$ to 
its restriction to $H_1$; note that since $\alpha$ is trivial 
on $H_2$, it is in  particular trivial on $H_3$ which is a subgroup of $H_2$.

\begin{theorem}\label{thm3.1}
The group 
$(H_1\cap H_2)/H_3$ 
which measures the difference between locally 
distinguished automorphic representations of $\SL_2(\A_E)$ and
globally distinguished automorphic representations of $\SL_2(\A_E)$
is trivial. 
\end{theorem}
\begin{proof}
We will prove that $(H_1\cap H_2)/H_3$ is trivial by proving that its character
group is trivial. From the analysis above, it suffices to prove
 the surjectivity of the natural map $$X(H_0/H_2) \rightarrow X(H_1/H_3).$$
Equivalently, we need to prove that 
 a character of $\A_F^\times/[F^\times (\A_F^\times \cap G_\Pi)]$, can be 
extended to a Gr\"ossencharacter of $\A_E^\times$ which is a self-twist
of $\widetilde{\Pi}$.

Let $\chi$ be a character of $\A_F^\times/[F^\times (\A_F^\times \cap G_\Pi)]$,
thought of as a character of $\A_F^\times/[\A_F^\times \cap G_\Pi]$.
Since $\A_F^\times/[\A_F^\times \cap G_\Pi]$ 
is a subgroup of the discrete group $\A_E^\times/G_\Pi$, there is a 
character $\widetilde{\chi}$ of $\A_E^\times$ trivial on $G_\Pi$ extending 
$\chi$. (We will eventually try to get one which is a Gr\"ossencharacter.)

Let ${\rm As}(\widetilde{\Pi})$ denote the Asai lift of a representation of $\GL_2(\A_E)$ to
$\GL_4(\A_F)$. By local considerations, it is clear that 
$${\rm As}(\widetilde{\Pi} \otimes \widetilde{\chi}) = {\rm As}(\widetilde{\Pi}) \otimes \chi.$$
Since $\widetilde{\chi}$ is trivial on $G_\Pi$, $\widetilde{\Pi} \otimes \widetilde{\chi} \cong 
\widetilde{\Pi}$, and hence 
$${\rm As}(\widetilde{\Pi} \otimes \widetilde{\chi})={\rm As}(\widetilde{\Pi}) =  {\rm As}(\widetilde{\Pi}) \otimes \chi.$$

Now let $\widehat{\chi}$ be a character of 
$\A_E^\times/E^\times$ extending the character $\chi$ of $\A_F^\times/F^\times$. 
We have,
\begin{eqnarray*} {\rm As}(\widetilde{\Pi}\otimes \widehat{\chi}) & \cong & {\rm As}(\widetilde{\Pi}) \otimes \chi \\
& \cong & {\rm As}(\widetilde{\Pi}).
\end{eqnarray*}

This implies that the Asai lifts of the two automorphic representations 
$\widetilde{\Pi}$ and $\widetilde{\Pi} \otimes \widehat{\chi}$ 
of $\GL_2(\A_E)$ to
$\GL_4(\A_F)$ are the same. Therefore we can use Proposition 5.4 to conclude a relationship between 
$\widetilde{\Pi}$ and $\widetilde{\Pi} \otimes \widehat{\chi}$.
Before we can apply this proposition, we need to check that the two representations 
$\widetilde{\Pi}$ and $\widetilde{\Pi} \otimes \widehat{\chi}$ have the same central characters (or the
determinants of the corresponding representations of $W_E$) 
restricted to $\A_F^\times$. But this follows as $\A_E^{\times 2} \subset G_\Pi$, and hence $\chi^2=1$.

By Proposition \ref{prop5.4}, there are two cases. 

{\bf Case 1:} 
There is a character $\chi_1$ of $\A_E^\times/E^\times$ trivial on $\A_F^\times/F^\times$ such that 
$$\widetilde{\Pi} \otimes \widehat{\chi} \cong \widetilde{\Pi} \otimes \chi_1.$$
Therefore, $\widetilde{\Pi} \cong \widetilde{\Pi} \otimes (\chi^{-1}_1\widehat{\chi})$. Since 
$\chi_1$ is trivial on $\A_F^\times/F^\times$, the character
$ \chi^{-1}_1\widehat{\chi}$ is an extension of $\chi$ to a Gr\"ossencharacter
on $\A_E^\times/E^\times$ which is a self-twist of $\widetilde{\Pi}$, proving the desired 
statement in this case. 

{\bf Case 2:} 
There is a character $\chi_1$ of $\A_E^\times/E^\times$ trivial on $\A_F^\times/F^\times$ such that 
\begin{eqnarray*} \widetilde{\Pi}\otimes \widehat{\chi} & \cong & \widetilde{\Pi}^\sigma \otimes \chi_1 \\
& \cong & \widetilde{\Pi}^\vee \otimes \chi_1 \\
& \cong & \widetilde{\Pi} \otimes (\chi_1\omega_{\widetilde{\Pi}}^{-1}),
\end{eqnarray*}
which again proves the desired statement since $\omega_{\tilde{\Pi}}$ restricted to $\A_F^\times$ is trivial.
 \end{proof}

\vspace{2mm}
\noindent{\bf Remark:} Although Asai lift naturally comes up in questions about 
distinguished representations for the pair $(\GL_2(\A_E),\GL_2(\A_F))$, its use in the
previous theorem is for an entirely different purpose: to prove that a certain
character of $\A_E^\times$ can be assumed to be a Gr\"ossencharacter when its restriction
to $\A_F^\times/F^\times$ is known to be a Gr\"ossencharacter. In this, the crucial property
of the Asai lift used is the fact that ${\rm As}(\Pi \otimes \chi) = {\rm As}(\Pi) \otimes 
\chi|_{\A_F^\times}$, so even if $\chi$ is not a Gr\"ossencharacter, since its
restriction to $\A_F^\times$ is, the Asai lift is an automorphic representation. 
This is then
combined with the knowledge about fibers of the Asai lift  to conclude 
that $\chi$, or a variant of it, is automorphic. Later when we deal with toric period integrals, 
we will use very similar arguments, using basechange map, for similar effect, which 
though does appear in toric period questions, is put to an unrelated use!

\section{Examples}
It may be useful to list all the possibilities for the groups which appear
in the previous section enumerated, which we do here.

According to the notation introduced in \cite{anand1}, \cite{anand2}, and 
the proof of the previous theorem, we have,
$$\begin{array}{rclcl}
X(H_1/H_3) & \subset & X_{\widetilde{\Pi}} & = & \left \{\chi \in \widehat{\A_F^\times/F^\times}| \widetilde{\Pi} 
{\rm ~~is~~} \chi-{\rm distinguished } \right \} \\
X(H_0/H_1H_2) & = & Y_{\widetilde{\Pi}} & = & \left \{\chi \in \widehat{\A_E^\times/E^\times}| \widetilde{\Pi}
\otimes \chi \cong \widetilde{\Pi}, \chi|_{\A_F^\times}=1 \right \} \\
X(H_0/H_2) & = & Z_{\widetilde{\Pi}} & = & \left \{\chi \in \widehat{\A_E^\times/E^\times}| \widetilde{\Pi}
\otimes \chi \cong \widetilde{\Pi} \right \}. 
\end{array}$$
Further, there is an isomorphism of groups $X_{\widetilde{\Pi}} 
\cong Y_{\widetilde{\Pi}}$. 

We  now enumerate all the possibilities for the groups 
$X_{\widetilde{\Pi}},Y_{\widetilde{\Pi}},Z_{\widetilde{\Pi}}$, 
and refer the reader to the proof of Theorem 6.9 in \cite{anand2}.

\begin{enumerate}
\item $\widetilde{\Pi}$ is not CM. In this case, $X_{\widetilde{\Pi}}=Y_{\widetilde{\Pi}} = Z_{\widetilde{\Pi}}= \{1\}$.

\item $\widetilde{\Pi}$ is CM by exactly one quadratic extension of $E$. In this case,
$X_{\widetilde{\Pi}}=Y_{\widetilde{\Pi}} = Z_{\widetilde{\Pi}}= \Z/2$, and therefore,
$$\frac{X_{\widetilde{\Pi}}}{Z_{\widetilde{\Pi}}/Y_{\widetilde{\Pi}}} = \Z/2.$$

\item $\widetilde{\Pi}$ is CM by three quadratic extensions of $E$, with exactly one
 Galois over $F$.
In this case,
$X_{\widetilde{\Pi}}=Y_{\widetilde{\Pi}}= \Z/2$, and $Z_{\widetilde{\Pi}} = \Z/2\oplus \Z/2$ and therefore,
$$\frac{X_{\widetilde{\Pi}}}{Z_{\widetilde{\Pi}}/Y_{\widetilde{\Pi}}} = \{e\}.$$

\item $\widetilde{\Pi}$ is CM by three quadratic extensions of $E$, all Galois over $F$.
In this case,
$X_{\widetilde{\Pi}}=Y_{\widetilde{\Pi}} = Z_{\widetilde{\Pi}}= \Z/2\oplus \Z/2$, and therefore,
$$\frac{X_{\widetilde{\Pi}}}{Z_{\widetilde{\Pi}}/Y_{\widetilde{\Pi}}} = \Z/2 \oplus \Z/2.$$

\end{enumerate}

\section{A more general situation}
In the context of distinction for the pair $(\GL_2(E),\GL_2(F))$, the most 
general pair of this kind that one could 
consider is $(\GL_1(D)(E), \GL_1(D)(F))$ where $D$ is a quaternion algebra over a 
number field $F$, and $E$ is a quadratic extension of $F$. In fact it was 
in the 
study of distinction property for this pair that relative trace formula 
was inaugurated by H. Jacquet in collaboration with K. Lai in \cite{lai} who dealt 
with only those quaternion algebras $D$ over $F$ for which 
$D\otimes_FE \cong  \M_2(E)$; the more general situation was considered in the
paper \cite{flicker3}. These papers prove that a cuspidal representation
$\Pi$ of $\GL_1(D)(\A_E)$ is globally distinguished by $\GL_1(D)(\A_F)$ 
if and only if $\Pi^{JL}$, the Jacquet-Langlands lift of $\Pi$ to
$\GL_2(\A_E)$,  is globally distinguished by $\GL_2(\A_F)$, together 
with the necessary local conditions at places of $F$ where $D$ is ramified, $E$ in inert, 
and $\Pi$ is a principal series representation. Thus 
distinction for these pairs also is dictated by the existence of 
a pole at $s=1$ of the Asai L-function.

Our work in the previous two sections  for the pair $(\SL_2(E),\SL_2(F))$ 
 was a consequence of this characterization
of distinction for $\GL_2(E)$ representations in terms of Asai $L$-function, and
an input on distinction for the pair $(\SL_2(E),\SL_2(F))$ in terms of
Whittaker model with respect to a character of $E$ trivial on $F$ which was
proved in  \cite{anand1} in the local case, and \cite{anand2} in the global case. 

In this section we consider the distinction property for the pair $(\SL_1(D)(\A_E), \SL_1(D)(\A_F))$ where $D$ is a quaternion algebra over a 
number field $F$, and $E$ is a quadratic extension of $F$. At places of $F$ where $E$ is inert  and $D$ is ramified, we will be dealing with
distinction properties for the pair $(\SL_2(E_v), \SL_1(D_v))$;  an added subtlety here is that the embedding  of $\SL_1(D_v)$ 
in $\SL_2(E_v)$ is unique only up to conjugation by $\GL_2(E_v)$, and there seems no preferred embedding of
$\SL_1(D_v)$ in $\SL_2(E_v)$. Thus one must keep in mind that the question about classifying representations of $\SL_2(E_v)$ distinguished by
$\SL_1(D_v)$ is meaningless unless there is a way of fixing an embedding of $\SL_1(D_v)$ inside $\SL_2(E_v)$. 

Recall that for the pair, $(\SL_2(\A_E),\SL_2(\A_F))$ our proof of the local-global
property depended on defining the groups, 
\begin{eqnarray*}
H_0 &=& \A_E^\times, \\
H_1 & = & \A_F^\times G_\pi, \\
H_2 & =& E^\times G_{\pi}, \\
H_3 & = &  F^\times G_{\pi},
\end{eqnarray*}
and noting the following:

\begin{enumerate}
\item The set $H_0 \cdot \pi$ is 
the set of $L$-packet of representations of $\SL_2(\A_E)$ determined by $\pi$.

\item The set $H_1 \cdot \pi$ is the set of locally distinguished representations
in the $L$-packet of $\SL_2(\A_E)$ determined by $\pi$.

\item The set $H_2 \cdot \pi$ is the set of automorphic representations
in the $L$-packet of $\SL_2(\A_E)$ determined by $\pi$.

\item The set $H_3 \cdot \pi$ is the set of globally distinguished representations
in the $L$-packet of $\SL_2(\A_E)$ determined by $\pi$.
\end{enumerate}

Then we proved, via considerations with the Asai lift, specially determination of the fibers
of the Asai lift, that $(H_1 \cap H_2)/H_3 = 1$,
for which we did not need the interpretation of 
$H_3 \cdot \pi$ as the set of globally distinguished representations
in the $L$-packet of $\SL_2(\A_E)$ determined by $\pi$; we only needed to know that members of $H_3 \cdot \pi$ are globally distinguished.

The groups $H_0,H_1,H_2,H_3$ were defined in the context of $(\SL_2({\mathbb A}_E),\SL_2({\mathbb A}_F))$ using the embedding of 
$E^\times$, or of $\A_E^\times$,  inside $\GL_2(\A_E)$ as the group of diagonal matrices:
$$\left (\begin{array}{cc} x & 0 \\ 0 & 1\end{array} \right ).$$
When dealing with $\SL_1(D)$, this diagonal subgroup does not make sense, but we can instead replace these by the image of
$D^\times(E)$, resp. $D^\times(E_v)$, resp. $D^\times(\A_E)$  in $E^\times$, resp. $E_v^\times$, resp $\A_E^\times$, via the reduced norm mapping. The group 
$G_\pi$ itself may be defined as the image of the reduced norm mapping of the
stabiliser in $D^\times(\A_E)$ of a representation $\pi$ of $\SL_1(D)(\A_E)$.  
Thus $H_0$  is  the image of $D^\times(\A_E)$
in $\A_E^\times$ under the reduced norm mapping, or rather this multiplied by $G_\pi$ inside $\A_E^\times$; similarly, 
$H_1 =\N m (D^\times(\A_F)) \cdot G_\pi \subset \A_E^\times$, and
$H_2 =\N m (D^\times(E)) \cdot G_\pi \subset \A_E^\times$.

In order to analyse the local-global question for $(\SL_1(D)(\A_E),\SL_1(D)(\A_F))$ inside an $L$-packet containing 
 a globally distinguished representation, we can adopt more or less the same strategy. But we see from Lemma \ref{lem7.1} below that $H_1 \cdot \pi$ does not capture all the locally distinguished representations, and thus we cannot proceed along exactly the same lines. However, we note that
$(H_1 \cap H_2)/H_3 = 1$, proves that locally distinguished representations of $(\SL_1(D)(\A_E),\SL_1(D)(\A_F))$ appearing in the restriction of a globally distinguished representation of 
\[D^+({\mathbb A}_E) = \{g \in D({\mathbb A}_E) | \det g \in {\mathbb A}_F^\times{\mathbb A}_E^{\times2}\}\]
are globally distinguished.

Let $\pi^+$ be an automorphic representation of $D^+({\mathbb A}_E)$ which is globally distinguished with respect to $\SL_1(D)(\A_F)$. Let $\pi$ be an automorphic representation of  $\SL_1(D)(\A_E)$ which comes in the restriction of $\pi^+$. Observe that:

\begin{enumerate}
\item The set $H_1 \cdot \pi$ are all the irreducible components of the restriction of $\pi^+$ to $\SL_1(D)(\A_E)$.
\item $H_2 \cdot \pi$ are the automorphic representations belonging to the $L$-packet of $\SL_1(D)(\A_E)$ determined by $\pi$.
\item $(H_1 \cap H_2)/H_3 = 1$.
\end{enumerate}


The statement (ii) is part of the work of Labesse-Langlands mentioned earlier too.

The proof of (iii) follows the same lines as given earlier for $\SL_2(\A_E)$ using the fibers of the Asai lift which this time can be 
considered to be lifting of automorphic representations of $D^\times({\A_E})
$ to $\GL_4(\A_F)$ via the intermediary of the Jacquet-Langlands
correspondence  to first $\GL_2(\A_E)$. We note that we also need to
use the standard local-global theorem for norms of quaternion division algebra: an element of $F^\times$ arises as a norm from
$D^\times$ if and only if it does locally at all places of $F$.

We summarize the above discussion in the following theorem.

\begin{theorem}\label{8.1}
Suppose $\pi^+$ is an irreducible cuspidal representation of 
$D^+(\A_E)$ which is globally $\SL_1(D)(\A_F)$ distinguished. Then the part of the $L$-packet of $\SL_2(\A_E)$ determined by
the restriction $\pi^+$ has local-global property for $\SL_1(D)(\A_F)$; more precisely,
automorphic representations of $\SL_1(D)(\A_E)$ contained in $\pi^+$ which are 
abstractly distinguished by $\SL_1(D)(\A_F)$, belong to one orbit under the action
of $D^\times(F)$. 
\end{theorem}

Now we would like to understand the local-global question for a cuspidal representation $\pi$ of $\SL_1(D)(\A_E)$. One important fact which went into our analysis of local-global distinction
for the pair $(\SL_2(\A_E),\SL_2(\A_F))$ was that if a representation
of $\GL_2^+(E_v)$ is distinguished by $\GL_2(F_v)$ 
then it must have a Whittaker
model for a character of $E_v$ which is trivial on $F_v$.  This has the corollary
that if two representations $\pi_1$ and $\pi_2$ of $\GL_2^+(E_v)$ belonging to the 
same $L$-packet are respectively $\omega_1$ and $\omega_2$ distinguished by $\GL_2(F_v)$ 
for two characters $\omega_1,\omega_2:F_v^\times \rightarrow \C^\times$, then $\pi_1 = \pi_2$
(although $\omega_1$ may not be the same as $\omega_2$). This is what allowed
us to prove that representations of $\SL_2(E_v)$  distinguished by $\SL_2(F_v)$ belonging to
one $L$-packet are in a single orbit for the action of $\GL_2(F_v)$. 
This property fails for the pair $(\SL_2(E_v),\SL_1(D_v))$ because of the following lemma.

\begin{lem}\label{lem7.1}
Let $K$ be a quadratic ramified extension of a non-archimedean local field $k$ of 
odd residue characteristic, and $D$ a quaternion division algebra over $k$.
Let $\mu$ be an unramified
character of $K^\times$ of order 4 with $\mu^2=\omega$. Then the  
principal series representation  $\pi={\rm Ps}(\mu,\mu\omega)$ of $\GL_2(K)$
decomposes as a sum of two irreducible representations $\pi^+$ and $\pi^-$
when restricted to $\GL_2^+(K)$ with $\pi^+$  spherical,
i.e., the one which contains a vector fixed under $\GL_2(\O_K)$. Fix an embedding of $D^\times$ in $\GL_2^+(K)$ such that  
$D^\times \subset  K^\times \cdot\GL_2(\O_K)$,  
then the trivial 
representation of $D^\times$ appears in $\pi^+$, and the 
character $\omega_K$ of order 2 of $k^\times$, considered as a character of $D^\times$ through the reduced norm mapping,
 appears in $\pi^-$.
\end{lem}
\begin{proof} It is easy to see that $D^\times$ operates transitively on ${\mathbb P}^1(K)$ such that the stabilizer of the point $\infty$ in ${\mathbb P}^1(K)$ is isomorphic to $K^\times$, 
hence by Mackey theory
we easily deduce that there are exactly two one dimensional representations of $D^\times$ contained in $\pi$, one the trivial character,
and the other which is the character $\omega_K$ of $k^\times$ considered as a character of $D^\times$. We need to see
which of these characters of $D^\times$ appear in which of the representations $\pi^+,\pi^-$.

Our first task will be  to construct an embedding of $D^\times$ in 
$  K^\times \cdot\GL_2(\O_K)$ (for $K$ quadratic ramified extension of $k$). For this we fix some notation. 

 Let $\varpi_K$ be a uniformizing element in $K$, and $\O_k, \O_K, \O_D$ 
be respectively the maximal compact subrings of $k, K, D$. Fix an embedding of $\iota: K \hookrightarrow D$.
We will consider $\O_D$ as a free rank 2 module over $\O_K$ from the right, and as an $\O_D$-module from the left. 
This gives an embedding $\O_D \hookrightarrow {\rm End}_{\O_K}(\O_D)$. 
Since $\O_D$ is invariant under conjugation
by $D^\times$, and hence by $K^\times$, left multiplication by $K^\times$ on $\O_D$ can be considered to be 
inner-conjugation by $K^\times$ up to an action of $K^\times$ on the right:
$$x \cdot \O_D = x\O_dx^{-1} \cdot x,$$
therefore $K^\times \subset D^\times$ is contained in $K^\times \cdot {\rm End}_{\O_K}(\O_D)$, 
and since $D^\times = K^\times\O_D^\times$, $D^\times$ is contained in $K^\times \cdot {\rm End}_{\O_K}(\O_D)$.

Observe that $\O_D$ comes equipped with a natural filtration consisting of two sided ideals:
$\O_D \supset \varpi_K\O_D \supset \varpi_K^2\O_D \supset \cdots$ such that the successive quotients are modules
for $\O_D/\varpi_K\O_D \cong {\F_{q^2}}$,  if $\F_q$ is the residue field of $k$. We thus have natural maps,
$$\O_D^\times \hookrightarrow {\rm Aut}_{\O_K}(\O_D) \longrightarrow {\rm Aut}_{\O_K}(\O_D/\varpi_K\O_D) = {\rm Aut}_{\F_q}(\F_{q^2}).$$

Under the composite map from $\O_D^\times$ to ${\rm Aut}_{\F_q}(\F_{q^2}) = \GL_2(\F_q)$, the image of $\O_D^\times$ is clearly 
$\F_{q^2}^\times$ acting on $\F_{q^2}$, giving rise to an embedding ${\F}_{q^2}^\times \hookrightarrow \GL_2(\F_q)$. Further, since
multiplication by $x \in K^\times$ on $\O_D$ on the left is up to a central element conjugation by $x$ on $\O_D$, the action
of $K^\times$ on $\O_D/\pi_K\O_D$ is an automorphism of algebras, i.e., an element of the Galois group of $\F_{q^2}$ over $\F_q$.
Thus the image of $D^\times$ is contained in the normalizer of $\F_{q^2}^\times$ inside $\GL_2(\F_q)$.

Since $ D^\times\subset K^\times \cdot \GL_2(\O_K)$,  given that  $\pi$ has trivial central character  and $\pi^+$ has a fixed vector under $\GL_2(\O_k)$, 
the trivial 
representation of $D^\times$ appears in $\pi^+$.  The representation 
$\pi^-$ is obtained from $\pi^+$ by conjugating by the matrix,
$$\left (\begin{array}{cc} \varpi_K & 0 \\ 0 & 1\end{array} \right ),$$
hence it is clear that $\pi^-$ has a subrepresentation on which $$\Gamma_0(\varpi) = \left \{ \left (\begin{array}{cc} a & b \\ c & d\end{array} \right ) \in \GL_2(\O_K) \bigg | \, \,\varpi_K | c    \right \}$$
acts trivially. This means that $\pi^-$ must contain the Steinberg representation
of $\GL_2(\F_q)$, where $\F_q$ is the residue field of $K$, as the Steinberg is the only 
non-trivial irreducible representation of $\PGL_2(\F_q)$ with a vector fixed under the group
of upper triangular matrices. Since the Steinberg representation contains all non-trivial characters of
${\F_{q^2}^\times}/{\F_q ^\times}$, 
the unique non-trivial character of ${\F_{q^2}^\times}/{\F_q ^\times}$ of order 2 appears in $\pi^-$. (This is where we use that $q$ is odd to ensure that
${\F_{q^2}^\times}/{\F_q ^\times}$ has a character of order 2.)
 Since the unique non-trivial character of 
${\F_{q^2}^\times}/{\F_q ^\times}$ of order 2 is left invariant by the normalizer of $\F_{q^2}^\times$ inside $\GL_2(\F_q)$, we conclude that there 
is a character of order 2 of 
$D^\times/k^\times$ appearing in $\pi^-$, which cannot be anything else but $\omega_K$. \end{proof}

Because of this lemma, the local-global property fails,  which we record in the following proposition.

\begin{prop} Let $E$ be a quadratic extension of a number field $F$, and $D$ a quaternion division algebra over $F$. Then there exists an
automorphic representation of $D^+(\A_E)$ which is abstractly distinguished by $\SL_1(D)(\A_F)$, but not globally distinguished in terms
of having nonzero period integral on this subgroup.
\end{prop}
\begin{proof} Let 
$\tilde{\pi}$ be a non-CM cuspidal representation of $D^\times(\A_E)$ with unramified principal series local components at many places where $E/F$ is ramified so that we are in the context of Lemma 8.2, and  the  restriction of $\tilde{\pi}$ to $D^+({\mathbb A}_E)$ has more than four direct summands which are 
abstractly distinguished by $\SL_1(D)(\A_F)$. Since a non-CM $L$-packet is stable, all these direct summands are automorphic as well. 
If all these representations where globally distinguished with respect to $\SL_1(D)({\mathbb A}_F)$, then in particular, they
will be globally $\omega$-distinguished with respect to $D^\times({\mathbb A}_F)$ for certain quadratic characters $\omega: \A_F^\times/F^\times
\rightarrow \Z/2$.  It would then follow that $\tilde{\pi}$ is distinguished by  $D^\times({\mathbb A}_F)$ for more than 
four Gr\"ossencharacters, which is not possible as global distinction is characterized in terms of the Asai lift of $\tilde{\pi}$ to 
$\GL_4(\A_F)$ to contain a Gr\"ossencharacter as a direct summand, and so $\tilde{\pi}$ can be $\omega$-distinguished for at most 
four Gr\"ossencharacters $\omega: \A_F^\times/F^\times \rightarrow \C^\times$. 
\end{proof}

On the positive side, we have the following result.

\begin{prop} Let $E$ be a quadratic extension of a number field $F$, and $D$ a quaternion division algebra over $F$. Let
$\Pi$ be an automorphic representation of $D^+(\A_E)$ which is abstractly distinguished by $D^\times(\A_F)$ by a Gr\"ossencharacter $\omega:
\A_F^\times/F^\times \rightarrow \C^\times$, then if $\Pi$ has a discrete series local component, it is  globally distinguished in terms
of having nonzero period integral on this subgroup with respect to the character $\omega$.
\end{prop}
\begin{proof} Let $\widetilde{\Pi}$ be an automorphic representation of $D^\times(\A_E)$ containing $\Pi$.
 We recall that by a theorem of Jacquet-Lai, cf. [JL85], if a cuspidal automorphic representation of $\GL_2(\A_E)$ 
has nonzero period integral on $D^\times(\A_F)$, then so is also the case for the period integral on $\GL_2(\A_F)$. There is analogous
local theorem proved by J. Hakim [Hak91], and by the second author in [Pr92]. Further, there is a more general global theorem
proved by Flicker and Hakim in [FH94] in which one uses the Jacquet-Langlands correspondence to go from automorphic representations
of $D^\times(\A_E)$ to automorphic representations of $\GL_2(\A_E)$. 

Given the  local results in the previous paragraph, it follows that $\widetilde{\Pi}^{JL}$ is abstractly $\omega$-distinguished 
with respect to 
$\GL_2(\A_F)$, and as recalled in the introduction, this means that $\widetilde{\Pi}^{JL}$ is  $\omega$-distinguished with respect to $\GL_2(\A_F)$ (this is where we use $\Pi$ having a discrete series local component, else the conclusion is 
either $\omega$-distinguished or $\omega\cdot \omega_{E/F}$-distinguished), and hence, by the global results of the previous
paragraph, $\widetilde{\Pi}$ is  $\omega$-distinguished 
with respect to $D^\times(\A_F)$. 

Now by the multiplicity one theorem about the space of $D^\times(F_v)$-invariant forms on an irreducible representation of $D^\times(E_v)$, 
it follows that $\Pi$ itself is $\omega$-distinguished, completing the proof of the proposition.
\end{proof}

Thus Proposition 8.4 says that the problem in the failure of the local-global principle in Proposition 8.3 is one of 
patching local characters of $F_v^\times$ into a Gr\"ossencharacter on $\A_F^\times$. We can capture this more precisely as follows.

To an automorphic representation $\Pi = \otimes \Pi_v$ of $D^\times(\A_E)$, define local groups ${\mathcal S}_v$ consisting of characters
$\omega_v$ of $F_v^\times$ such that $\Pi_v$ is $\omega_v$-distinguished with respect to the subgroup $D^\times(F_v)$. We know that
${\mathcal S}_v$ is a finite set consisting of at the most 4 elements, and that for most places $v$ of $F$, there is an
unramified character in ${\mathcal S}_v$, and there are at the most two unramified characters in ${\mathcal S}_v$ which if there are two 
are twists of each other by $\omega_{E_v/F_v}$. So the previous proposition can also be stated as saying that an
automorphic representation $\Pi$ of $D^\times(\A_E)$ is globally distinguished by $\SL_1(D)(\A_F)$ if and only if  there is a 
Gr\"ossencharacter $\omega = \prod \omega_v: \A_F^\times/F^\times \rightarrow \C^\times$ such that $\omega_v$ belongs to ${\mathcal S}_v$
for all places $v$ of $F$.

\section{Distinction of some member in an $L$-packet for the toric period}
In this section we prove theorem 1.4 which we recall for the convenience of the reader.

\begin{theorem}\label{3} Let $D$ be a quaternion  algebra over a number field $F$, with $E$ a quadratic subfield of $D$.
Let $\Pi = \otimes_v \Pi_v$ be a cuspidal representation of $\SL_1(D)({\mathbb A}_F)$ 
with at least one square integrable component at a place $v_0$ of $F$ which we assume is of odd residue characteristic if $E$ is inert 
and $D$ split at $v_0$.  If each $\Pi_v$ is distinguished with respect to 
$E_v^1$, then there is a cuspidal representation in the $L$-packet of $\Pi$ 
which is distinguished with respect to ${\mathbb A}_E^1$.
\end{theorem}

\begin{proof} Let 
$\widetilde{\Pi}=\otimes \widetilde{\Pi}_v$ be a cuspidal representation of $D^\times({\mathbb A}_F)$ 
containing $\Pi = \otimes \Pi_v$. We assume that each $\Pi_v$ is distinguished with respect to $E_v^1$, 
the group of norm one elements of $E_v^{\times}$. 
We assume, without loss of generality, that the central character of $\widetilde{\Pi}$ is trivial. 
Since $\Pi_v$ and hence  $\widetilde{\Pi}_v$ 
is distinguished with respect to $E_v^1$, $\widetilde{\Pi}_v$ is 
 $\alpha_v$-distinguished 
with respect to $E_v^{\times}$, 
where $\alpha_v$ is a character of $\frac{E_v^{\times}}{F_v^{\times}E_v^{1}}$, hence there 
is a quadratic character $\beta_v$ of $F_v^{\times}$ such that $\widetilde{\Pi}_v$ is $\beta_v \circ \N m$
distinguished with respect to $E_v^{\times}$. Since being distinguished is no condition at places $v$ of $F$ where $\widetilde{\Pi}_v$ 
is a principal series representation ---so at all but finitely many places of $F$--- we can assume that $\widetilde{\Pi}_v$ is $\beta_v \circ \N m
$-distinguished
at all places $v$ of $F$ for a Gr\"ossencharacter $\beta$ of ${\mathbb A}_F^{\times}$ with $\beta^2=1$.

The proof of Theorem \ref{3} will depend on two technical lemmas, 
one local and the other global. The local lemma allows one to twist a 
representation $\widetilde{\Pi}_0$ of $\GL_2(F_0)$ by a 
quadratic character $\chi_0$ 
to change epsilon factor $\epsilon(\widetilde{\Pi}_0)$ to 
$\epsilon(\widetilde{\Pi}_0 \otimes \chi_0)$ 
so that the global epsilon factor $\epsilon(\widetilde{\Pi} \otimes \chi)$ can be 
assumed to be 1 if the original $\epsilon(\widetilde{\Pi})$ was $-1$; this then allows one to appeal to Conjecture \ref{conj}  about simultaneous nonvanishing of two $L$-values. An
added subtlety that we must deal with is that in changing the sign of 
$\epsilon(\widetilde{\Pi}_0)$ to $\epsilon(\widetilde{\Pi}_0 \otimes \chi_0)$, such quadratic characters
must appear in the representation $\widetilde{\Pi}_0$ at that place, so by the theorem of Saito-Tunnell, some other epsilons must be controlled. This is where 
existence of a discrete series component of the automorphic representation is 
used.

\begin{lem}\label{lemma8.1} Let $\widetilde{\Pi} = \otimes \widetilde{\Pi}_v$ be an automorphic representation of 
$D^\times({\mathbb A}_F)$ with trivial central character. Let $E$ be a quadratic extension of $F$ contained in $D$ 
with $\omega_E = \prod \omega_v$ the associated Gr\"ossencharacter of ${\mathbb A}_F^{\times}$. Assume that 
$\widetilde{\Pi}$ has at least 
one square integrable component, say at $v_0$ which we assume is of odd residue characteristic if $E$ is inert at $v_0$, and $D$ is split at $v_0$. 
Assume that $\widetilde{\Pi}$ is abstractly distinguished by the character $\beta \circ \N m$ of 
${\mathbb A}_E^\times$ for a quadratic Gr\"ossencharacter 
$\beta$ of ${\mathbb A}^\times_F$. 
Then there is a Gr\"ossencharacter $\eta$ of ${\mathbb A}_F^{\times}$ with $\eta^2=1$, 
such that 
\begin{eqnarray}\label{id1}
\epsilon(\widetilde{\Pi}\otimes \eta)=1=\epsilon(\widetilde{\Pi}\otimes \omega_E \eta),
\end{eqnarray}
and furthermore, 
\begin{eqnarray}\label{id2}
\epsilon(\widetilde{\Pi}_v\otimes \eta_v) \epsilon(\widetilde{\Pi}_v\otimes \omega_v\eta_v)=
\epsilon(\widetilde{\Pi}_v \otimes \beta_v) \epsilon(\widetilde{\Pi}_v\otimes \omega_v\beta_v)
\end{eqnarray}
for all $v$. 
Moreover, $\eta$ can be made to agree with $\beta$ at finitely many prescribed places other than $v_0$.
\end{lem}

The proof of this lemma will depend on a  local lemma that we come to presently, but before we do that, we fix some notation.
Some of this  notation as well as the proofs that follow are due to one of the referees to this paper.

Let $\tilde{\pi}_0$ be a discrete series representation of $\PGL_2(F_0)$.
For $\nu = \pm 1$, let 
$$X^{\nu} = \left \{ \chi: F^\times \rightarrow \Z/2 \,\,| \,\,\epsilon(\tilde{\pi}_0 
\otimes \chi )= \nu \chi(-1)\epsilon(\tilde{\pi}_0) \right \}.$$

Since $\tilde{\pi}_0$ is a discrete series representation, it is known that both $X^+$ and $X^-$ are non-empty \cite{waldspurger1}.

For $\nu,\nu' = \pm 1$, let 
$$X^{\nu,\nu'} = \left \{ \chi \in X^\nu \,\, \big | \,\, \omega_0 \chi \in X^{\nu'} \right \}.$$
Clearly, the map $\chi \rightarrow \omega_0 \chi$ is a bijection of $X^{\nu,\nu'}$ onto $X^{\nu',\nu}$. Therefore if 
$X^{+-} \not = \emptyset$, 
then also $X^{-+} \not = \emptyset$. In this case, if $\gamma_0 \in X^{+-}$, and $\delta_0 \in X^{-+}$, then the pair
$(\gamma_0,\delta_0)$ satisfies the two conditions in the statement of the lemma. If  $X^{+-} = \emptyset$, 
then also $X^{-+}= \emptyset$, in which case, condition (i) implies condition (ii). Note also that by the theorem of Saito-Tunnel, 
a character $\alpha$ belonging to $X^{\nu,\nu'}$, appears in the restriction of $\tilde{\pi}_0$ to $E_0^\times$ if $\nu \nu'=1$,
and belongs to the restriction of the representation $\tilde{\pi}_0^{JL}$ of $D^\times$ to $E_0^\times$ if $\nu \nu'=-1$.

\begin{lem}\label{lemma8.2} In the notation introduced above, if $X^{+-}$ is non-empty, so is $X^{-+}$, and conversely. 
In odd residue characteristic, if $X^{++}$ is non-empty, so is $X^{--}$, and conversely. 
\end{lem}

\begin{proof} As observed above, $\eta \rightarrow \omega_0 \eta$ gives a bijection of 
$X^{+-}$ with $X^{-+}$, so if $X^{+-}$ is non-empty, so is $X^{-+}$, and conversely. 

The second part of the lemma is subtler, and we can offer a proof only in odd residue characteristic where it is known that
the number of quadratic characters of $F_0^\times$ is 4, and further any irreducible admissible representation $\widetilde{\pi}$ 
of $\GL_2(F_0)$ has a self-twist by a non-trivial character $\mu$ of order 2. This implies that the sets
$X^{\nu,\nu'}$ are stabilized by multiplication by $\mu$, and hence their cardinalities are even integers. Given that
$X^{++}, X^{+-},X^{--},X^{-+}$ 
are disjoint sets of total cardinality 4 with cardinalities of $X^{+-}$ and $X^{-+}$ equal, we  easily deduce 
 that it is not possible for
$X^{++}$ to be non-empty but $X^{--}$ to be empty, and conversely; here we  also use the fact that  
$X^+, X^-$ are both known to be non-empty. 
\end{proof}

\noindent{\bf Remark:} We are unable to prove the second part of the lemma above in residue characteristic 2 one way or the other which seems like an
interesting {\it Exercice dyadiques}.

\begin{lem}\label{lemma8.3}
Let $E_0$ be a separable quadratic algebra over a local field $F_0$, with $\omega_0$ as the
corresponding character of $F_0^{\times}$ with $\omega_0^2=1$. Let $\widetilde{\pi}_0$ be an irreducible
admissible discrete series representation of $\PGL_2(F_0)$; if $E_0$ is a quadratic field extension of $F_0$, assume that 
it is of odd residue characteristic. 
Let $\beta_0$ be a quadratic character of $F_0^\times$. Then there exists  a quadratic
character $\gamma_0$ of $F_0^{\times}$ with,
 \[\frac{\epsilon(\widetilde{\pi}_0 \otimes \gamma_0)}{\gamma_0(-1)}   =  -\frac{\epsilon(\widetilde{\pi}_0 \otimes \beta_0)}{\beta_0(-1)},
\]  {and,} 

 $$\epsilon({\widetilde{\pi}}_0\otimes \gamma_0) \epsilon({\widetilde{\pi}}_0\otimes \omega_0\gamma_0)  =  
\epsilon({\widetilde{\pi}}_0 \otimes \beta_0) \epsilon({\widetilde{\pi}}_0\otimes \omega_0\beta_0).$$

\end{lem}

\begin{proof} If $\beta_0$ belongs to $X^{+-}$, then choose $\gamma_0$ from the non-empty set $X^{-+}$, and conversely. If $\beta_0$ belongs to $X^{--}$, 
then choose $\gamma_0$ from the non-empty set $X^{++}$, and conversely. (The non-emptiness of the sets involved is the conclusion of the previous lemma,
which needs an appeal to odd residue characteristic.)
\end{proof} 
\vspace{2mm}

\noindent{\bf Proof of Lemma \ref{lemma8.1}:} Since the representation $\widetilde{\Pi}_v$ of $D_v^\times$ 
is $\beta_v \circ \N m$ distinguished for the subgroup $E_v^\times$, by the Saito-Tunnell theorem, we have,
$$\epsilon(\widetilde{\Pi}_v \otimes \beta_v) \epsilon(\widetilde{\Pi}_v\otimes \omega_v\beta_v) = \omega_v(-1) \omega_{D_v}(-1),$$
where $\omega_{D_v}(-1) = -1$ if $D_v$ is ramified, and  $\omega_{D_v}(-1) = 1$ if $D_v \cong M_2(F_v)$. It follows that,
$$\epsilon(\widetilde{\Pi} \otimes \beta)   \epsilon(\widetilde{\Pi} \otimes \beta \omega)   
= \prod_v \omega_{D_v}(-1) = 1,$$
the last equality following from the fact that   the number of ramified primes 
of $D$ are even in number.

Therefore, if $\epsilon(\widetilde{\Pi} \otimes \beta)=1$,  then so is $\epsilon(\widetilde{\Pi} \otimes \beta \omega)$,
and $\eta=\beta$  has
all the desired properties to apply Conjecture \ref{conj} .

If $\epsilon(\widetilde{\Pi} \otimes \beta)=-1$, 
we will use the fact that $\widetilde{\Pi}$ has a square integrable component at $v_0$ to modify $\beta$ to
construct $\eta$ such that $\epsilon(\widetilde{\Pi} \otimes \eta)=1 = \epsilon(\widetilde{\Pi} \otimes \eta \omega) $, 
with
$$\epsilon(\widetilde{\Pi}_v\otimes \eta_v) \epsilon(\widetilde{\Pi}_v\otimes \omega_v\eta_v)=
\epsilon(\widetilde{\Pi}_v \otimes \beta_v) \epsilon(\widetilde{\Pi}_v\otimes \omega_v\beta_v)$$
for all $v$.

Let $\gamma$ be a quadratic Gr\"ossencharacter of ${\mathbb A}_F^{\times}$ with $\gamma_{v_0} = \gamma_0$ as in Lemma \ref{lemma8.3},
and which at the other  places $v$ of $F$ where either $D$ or $\widetilde{\Pi}$ is ramified 
is $\beta_v$ (and no constraints outside the ramified primes of 
$\widetilde{\Pi}$). 
By a well-known calculation
about the epsilon factor of a principal series representations of 
$\PGL_2(F_v)$, it follows that 
$$\frac{\epsilon({\widetilde{\Pi}_v}\otimes \chi)}{\chi(-1)}=\epsilon({\widetilde{\Pi}}_v),$$
for $\Pi_v$ a principal series representation of $\PGL_2(F_v)$, and  $\chi$ any character of $F_v^\times$ with $\chi^2=1$. Therefore we have,
$$ \epsilon(\widetilde{\Pi} \otimes \gamma) = \prod_v 
\frac{\epsilon({\widetilde{\Pi}_v}\otimes \gamma_v)}{\gamma_v(-1)}= 
-\frac{\epsilon({\widetilde{\Pi}_0}\otimes \beta_0)}{\beta_0(-1)}
\prod_{v \not = v_0} \frac{\epsilon({\widetilde{\Pi}_v}\otimes \beta_v)}{\beta_v(-1)} = -\epsilon(\widetilde{\Pi} \otimes \beta),$$
proving Lemma \ref{lemma8.1}.

\vspace{4mm}
\noindent{\bf (Proof of Theorem \ref{3}):}
Appealing  to Conjecture \ref{conj},  we get a  Gr\"ossencharacter, say $\eta^\prime$, of ${\mathbb A}_F^{\times}$ with $(\eta')^2=1$, such that 
\begin{enumerate}
\item
$L(\frac{1}{2},\widetilde{\Pi}\otimes \eta^\prime) 
\neq 0 \neq L(\frac{1}{2},\widetilde{\Pi}\otimes \omega \eta^\prime) $
\item
$\eta^\prime$ agrees with $\eta$ at all the  places $S$ of $F$ containing the infinite places of $F$, and the places of $F$ where 
$\widetilde{\Pi}$ or $D$ is ramified.
\end{enumerate}
Given this,
equation (\ref{id2}) of Lemma \ref{lemma8.1} 
continues to be satisfied with $\eta^\prime$ instead of $\eta$ at all places $v$ of $F$ since outside of $S$, there is no condition as can be easily checked. 
Thus, $\widetilde{\Pi}\otimes \eta^\prime$ is distinguished with respect to $E_v^{\times}$ at all the places $v$. Since $L(\frac{1}{2},\widetilde{\Pi}\otimes \eta^\prime) \neq 0 \neq L(\frac{1}{2},\widetilde{\Pi}\otimes 
\omega \eta^\prime) $,  the nonvanishing of the toric period 
on $\widetilde{\Pi}\otimes \eta^\prime$ follows by the work of Waldspurger. This is enough to conclude that 
there is a member in the $L$-packet of $\Pi$ on which 
the ${\mathbb A}_E^1$-period integral is nonvanishing by Proposition \ref{global2}. This finishes the proof of Theorem \ref{3}. \end{proof}

\noindent{\bf Remark:}
The previous arguments work as well for 
 the split toric period of a cuspidal representation 
of $\SL_2({\mathbb A}_F)$. In fact, in this case, since $\omega_{E/F}=1$,
there are not two $L$-values to control, but a single one, 
whose nonvanishing is the main theorem of the paper of 
Friedberg and Hoffstein \cite{hoffstein}, so we do not need to
resort to Conjecture \ref{conj}  in the split toric case, and we get an unconditional theorem.
Further, in this case local-global 
principle holds true for individual automorphic representations, 
since automorphic representations (in an $L$-packet)
are all $F^{\times}$ conjugate of each 
other and therefore if  period integral is nonzero on one automorphic representation,
it is nonzero on any other automorphic member of the $L$-packet. 

\section{Local-global principle for toric period}
Let $E$ be a quadratic extension of a number field $F$. Fix an embedding of $E^\times$
in $\GL_2(F)$, and hence an embedding of $E^1$ into $\SL_2(F)$.  
Let $\Pi = \otimes \Pi_v$ be an automorphic representation of $\SL_2(\A_F)$.
The group $\A_F^\times$ sitting inside $\GL_2(\A_F)$ as 
$$\left (\begin{array}{cc} x & 0 \\ 0 & 1\end{array} \right ),$$
operates on $\SL_2(\A_F)$ via conjugation action, and therefore on the 
set of isomorphism classes of representations of $\SL_2(\A_F)$. The orbit 
of $\Pi$ under the action of $\A_F^\times$ is precisely the set of global 
$L$-packet of representations of $\SL_2(\A_F)$ containing $\Pi$. Let
$G_\Pi \subset \A_F^\times $,  $G_\Pi= \prod G_v$,
 be the stabilizer of the representation $\Pi 
=\otimes \Pi_v$, where $G_v$ is the stabilizer inside $F_v^\times$ 
of the representation $\Pi_v$.

The action of $F^\times$ on $\SL_2(\A_F)$ 
is transitive on the set of 
automorphic representations of $\SL_2(\A_F)$ contained in the global $L$-packet determined by $\Pi$.
Clearly if $\Pi$ has a nonzero period integral
on the given embedding of $E^1(\A_F)$ in $\SL_2(\A_F)$, then so will all its conjugates under $\N(E^\times)$. 
If we can prove that these are the only automorphic representations of $\SL_2(\A_F)$ which have
local periods with respect to $E^1(F_v)$ for all places $v$ of $F$, we will have proved the
local-global principle for toric periods. However, the proof in the toric case 
will not be so simple, and will depend on using another related group 
$\GL^+_2(\A_F)$, defined as follows:
$$\GL^+_2(\A_F) = \{g \in \GL_2(\A_F) | \det g \in \N(\A_E^\times) \}.$$

We will prove that 
the part of the $L$-packet of $\SL_2(\A_F)$ determined by
the restriction of an irreducible automorphic representation $\Pi^+$ of 
$\GL^+_2(\A_F)$
has local-global property if $\Pi^+$ is globally distinguished by a quadratic
character of $\A_F^\times$.

\begin{theorem}\label{9.1}
Suppose $\Pi^+$ is an irreducible cuspidal representation of 
$\GL^+_2(\A_F)$ which is globally $\A_E^\times$ distinguished by a quadratic
character $\omega$ of $\A_F^\times/F^\times$, i.e., by the character $\omega \circ \N$ of
$\A_E^\times/E^\times$. Then the part of the $L$-packet of $\SL_2(\A_F)$ determined by
the restriction of an irreducible automorphic representation $\Pi^+$ of 
$\GL^+_2(\A_F)$ has local-global property for $\A^1_E/E^1 \subset \SL_2(F) \backslash \SL_2(\A_F)$.
\end{theorem}

\begin{proof} Define groups analogous to the ones defined  
in the previous section:
\begin{eqnarray*}
H_0 &=& \A_F^\times, \\
H_1 & = & \N(\A_E^\times) G_\Pi, \\
H_2 & =& F^\times G_{\Pi}, \\
H_3 & = &  {\N(E^\times)} G_{\Pi}.
\end{eqnarray*} Let $\Pi$ be an automorphic representation of $\SL_2(\A_F)$ 
contained in $\Pi^+ = \otimes_v \Pi_v^+$ which is globally distinguished 
by $\A_E^1$; its 
existence follows  from Proposition \ref{global1}. Clearly, automorphic representations of 
$\SL_2(\A_F)$ of the form $H_3 \cdot \Pi$ are globally distinguished by 
$\A_E^1$, whereas representations of the form $H_1 \cdot \Pi$ are 
all the irreducible components of $\Pi^+$ restricted to $\SL_2(\A_F)$, and
among these, representations of the form $H_2 \cdot \Pi$ are automorphic.
Thus the following result  proves the theorem. \end{proof}

\begin{theorem}\label{11.2}
The group 
$(H_1\cap H_2)/H_3$  is trivial. 
\end{theorem}
\begin{proof}
We will prove that $(H_1\cap H_2)/H_3$ is trivial by proving that its character
group is trivial. 
Noting that 
$(H_1\cap H_2)/H_3$ 
is nothing but the kernel of the map,
$$H_1/H_3 \rightarrow H_0/H_2,$$
the character group of $(H_1\cap H_2)/H_3$ is the cokernel of
 the natural map $$X(H_0/H_2) \rightarrow X(H_1/H_3).$$
Therefore to prove the theorem, it suffices to prove
 the surjectivity of the natural map $$X(H_0/H_2) \rightarrow X(H_1/H_3).$$
Equivalently, we need to prove that 
 a character of $(\N(\A_E^\times)G_\Pi)/[\N(E^\times) G_\Pi]$, can be 
extended to a Gr\"ossencharacter of $\A_F^\times$ which is a self-twist
of $\tilde{\Pi}$.

Since $G_\Pi$ and hence $\N(E^\times) G_\Pi$ is an open subgroup of $\A_F^\times$, 
$\A_F^\times/[\N(E^\times) G_\Pi]$ is a discrete group, hence  a character $\chi$ of $[\N(\A_E^\times)G_\Pi]/[\N(E^\times) G_\Pi]$
can be thought of as a character of $\A_F^\times/[\N(E^\times) G_\Pi]$, so that
$\widetilde{\Pi} \cong \widetilde{\Pi} \otimes \chi$.
Our aim is to  eventually get one which is a Gr\"ossencharacter.

Let ${\rm BC}(\widetilde{\Pi})$ denote the base change lift of the  representation $\widetilde{\Pi}$ 
of $\GL_2(\A_F)$ to $\GL_2(\A_E)$. By local considerations, it is clear that 
\begin{eqnarray}\label{(3)}  
{\rm BC}(\widetilde{\Pi}) \cong  {\rm BC}(\widetilde{\Pi} \otimes {\chi}) \cong 
 {\rm BC}(\widetilde{\Pi}) \otimes \chi\circ \N.  
\end{eqnarray}

Note that although we do not know that $\chi$ is a  Gr\"ossencharacter 
on $\A_F^\times,$ but since it is trivial on $\N(E^\times)$, the character
$\chi \circ \N$ 
of $\A_E^\times$ is a  
Gr\"ossencharacter on $\A_E^\times/E^\times$. 
Further, the Gr\"ossencharacter 
$\chi \circ \N$ on $\A_E^\times/E^\times$ is 
naturally Galois-invariant. Therefore, the Gr\"ossencharacter 
$\chi \circ \N$ 
on $\A_E^\times/E^\times$ can be descended to a 
Gr\"ossencharacter, say $\mu$ on $\A_F^\times/F^\times$, i.e.,
$$\chi \circ \N = \mu \circ \N.$$ 
So the equation $(3)$ can be rewritten as,
\begin{eqnarray}\label{(4)}
  {\rm BC}(\widetilde{\Pi}) \cong  {\rm BC}(\widetilde{\Pi} \otimes {\chi}) 
\cong 
 {\rm BC}(\widetilde{\Pi}) \otimes \chi\circ \N  \cong {\rm BC}(\widetilde{\Pi}) \otimes \mu\circ \N \cong 
{\rm BC}(\widetilde{\Pi} \otimes {\mu}).  
\end{eqnarray}
This gives,
$$  {\rm BC}(\widetilde{\Pi}) \cong {\rm BC}(\widetilde{\Pi} \otimes {\mu}). $$

Just like the previous case dealing with Asai lift, appealing 
now to --- this time a much better known theorem --- about fibers of the base change map,
we find that either,
\begin{enumerate}
\item $\widetilde{\Pi}  \cong  \widetilde{\Pi} \otimes \mu$, or 

\item $\widetilde{\Pi} \otimes \omega_{E/F} \cong  \widetilde{\Pi} \otimes \mu.$
\end{enumerate}

In case (i), the character $\mu$ is trivial on $G_\Pi$ (by the very definition of $G_\Pi$),
and since $\chi \circ \N = \mu \circ \N,$ we find that $\chi$ and $\mu$ are the same on the subgroup
$\N(\A_E^\times)$ of $\A_F^\times$, therefore the character $\mu$  on $\A_F^\times/F^\times$
is the desired extension of the character
$\chi$ initially defined on $(\N(\A_E^\times)G_\Pi)/[\N(E^\times) G_\Pi]$.

In case (ii), the character $\mu\omega_{E/F}$ is trivial on $G_\Pi$,
and since $\chi \circ \N = \mu \circ \N,$ we find that $\chi$ and $\mu\omega_{E/F}$ 
are the same on the subgroup
$\N(\A_E^\times)$ of $\A_F^\times$, therefore the character $\mu\omega_{E/F}$ on 
$\A_F^\times/F^\times$
is the desired extension of the character
$\chi$ initially defined on $(\N(\A_E^\times)G_\Pi)/[\N(E^\times) G_\Pi]$. \end{proof}

\noindent{\bf Remark:} Theorem \ref{9.1} holds true in the 
analogous division algebra case, and the proof is the same after we have noted that the group $H_0 = \A_F^\times$ which is used as a subgroup
of $\GL_2(\A_F)$ can also be treated as a quotient group via the determinant map, and then it once again operates
on $\SL_2(\A_F)$ via conjugation, well-defined up to inner-automorphisms, so also on its representations; this then allows
one to define $H_0$ for $D^\times(\A_F)$ as the image of the reduced norm mapping, and $H_1,H_2,H_3$ as subgroups of this norm
mapping. The appeal to basechange from $\GL_2(\A_F)$ to $\GL_2(\A_E)$ can be done using Jacquet-Langlands correspondence
from automorphic representations of $D^\times(A_F)$ to automorphic representations of $\GL_2(\A_F)$ and then to $\GL_2(\A_E)$.

\vspace{2mm}

The strategy in the present paper to come to grips with those automorphic representations 
of $\SL_2(\A_F)$ in a given global $L$-packet which have nonzero period integral 
for a given embedding of $E^1(\A_F)$ inside $\SL_2(\A_F)$ is to prove that such global
packets which have no local obstructions for non-vanishing are conjugate to each other by an
element of $\N E^\times \subset F^\times$ instead of just being conjugate by $F^\times$ which is the case as they belong to 
the same $L$-packet. The following lemma suggests that this strategy will
not succeed in the presence of certain principal series components, which one may call {\it supersingular primes},
being analogues of supersingular primes for elliptic curves.

\begin{lem}Let $K$ be a quadratic unramified extension of a local 
field $k$ of odd residue characteristic. Let $\mu$ be an unramified
character of $k^\times$ of order 4 with $\mu^2=\omega_{K/k}$. Then the  
principal series representation  $\pi={\rm Ps}(\mu,\mu\omega_{K/k})$ of $\GL_2(k)$
decomposes as a sum of two irreducible representations $\pi^+$ and $\pi^-$
when restricted to $\GL_2^+(k)$ in which $\pi^+$ is the one which is spherical,
i.e., contains a vector fixed under $\GL_2(\O_k)$. Fix an embedding of $K^\times$ in $\GL^+_2(k)$ such that 
$K^\times \subset k^\times \cdot \GL_2(\O_k)$, 
then the trivial 
representation of $K^\times$ appears in $\pi^+$, and the 
ramified character of order 2 of $K^\times/k^\times$ appears in $\pi^-$.
\end{lem}
\begin{proof} Let $\varpi$ be a uniformizing element in $k$, and $\O_k, \O_K$ 
be respectively the maximal compact subrings of $k$ and $K$. 
Since $ K^\times\subset k^\times \cdot \GL_2(\O_k)$,  $\pi^+$ trivial central character,  and $\pi^+$ has a vector fixed under $\GL_2(\O_k)$, 
the trivial 
representation of $K^\times$ appears in $\pi^+$. The representation 
$\pi^-$ is obtained from $\pi^+$ by conjugating by the matrix,
$$\left (\begin{array}{cc} \varpi & 0 \\ 0 & 1\end{array} \right ),$$
hence it is clear that $\pi^-$ has a subrepresentation on which $$\Gamma_0(\varpi) = \left \{ \left (\begin{array}{cc} a & b \\ c & d\end{array} \right ) \in \GL_2(\O_k) \bigg | \, \,\varpi | c    \right \}$$
acts trivially. This means that $\pi^-$ must contain the Steinberg representation
of $\GL_2(\F_q)$ where $\F_q$ is the residue field of $k$ as the Steinberg is the only 
non-trivial irreducible representation of $\PGL_2(\F_q)$ with a fixed vector under the group
of upper triangular matrices. Since the Steinberg representation contains all non-trivial characters of
${\F_{q^2}^\times}/{\F_q ^\times}$, the
conclusion about the ramified character of order 2 of 
$K^\times/k^\times$ appearing in $\pi^-$ follows. \end{proof}

It should be noted that for an automorphic representation $\widetilde{\Pi}$ of 
$\GL_2(\A_F)$ of trivial central character, at a place $v_0$ of odd 
residue characteristic where  $\widetilde{\Pi}_0$ is unramified, the representation
$\widetilde{\Pi}_0$ 
decomposes when restricted to $\GL_2^+(F_0)$ if and only if 
$\widetilde{\Pi}_0$ is as in the previous lemma, i.e., the  principal series
representation $\pi={\rm Ps}(\mu,\mu\omega_{K/k})$ with $\mu^2 = \omega_{K/k}$. 
These are what are called supersingular primes in the classical language, and are interpreted by vanishing of the Fourier-coefficient: $a_v=0$. It is expected
that for non-CM modular forms of weight $\geq 4$, there are only finitely many 
supersingular primes (for arbitrary $F$); for example, a famous
 conjecture of Lehmer
asserts that there are no supersingular primes for the Ramanujan Delta function.
Thus the following theorem is not without content, although its applicability
at the moment is only theoretical; besides, its proof also depends on
Conjecture \ref{conj}  about simultaneous nonvanishing of $L$-values.

\begin{theorem}\label{thm11.4}
Let $\widetilde{\Pi}$ be  an automorphic representation 
of $\GL_2(\A_F)$ with trivial central character, and with 
a discrete series local component at an odd place, say $v_0$, and only 
finitely many supersingular primes.  Then 
any automorphic representation $\Pi^+$ of $\GL_2^+(\A_F)$ contained in $\widetilde{\Pi}$ which is 
locally distinguished by $E^1(\A_F)$
is globally $\lambda$-distinguished for a character $\lambda$ 
of ${\A_E^\times/E^\times\A_F^\times}$ of order 2, and hence by Theorem \ref{9.1}, local-global
principle holds for automorphic representations of $\SL_2(\A_F)$ contained in 
$\widetilde{\Pi}$ for the subgroup $\A^1_E$.
\end{theorem}
\begin{proof} 
Let $S$ be a finite set of places containing all 
ramified places of $\widetilde{\Pi}$, places of residue characteristic 2, 
as well as infinite places, and all the supersingular primes which we are assuming is a finite set. By the remarks above, the representation 
$\widetilde{\Pi}_v$ remains irreducible when restricted to $\GL_2^+(F_v)$ for places $v$ outside $S$. 
Since $\Pi^+$ is locally distinguished by $E^1(\A_F)$, it is $\lambda_v$-distinguished
for some characters $\lambda_v$ of $F_v^\times$ of order $\leq 2$. 
Globalize
these characters $\lambda_v$ for $v$ in $S$ to a quadratic character 
$\lambda$ of  $\A_F^\times/F^\times$ for which we then know that $\Pi^+$ is locally
$\lambda$-distinguished at all places of $F$ because of an easy observation that an irreducible
principal series representation of $\PGL_2(F_v)$ is $\lambda_v$-distinguished
for any quadratic character $\lambda_v$ of $E_v^\times/F_v^\times$.  
 By Lemma \ref{lemma8.1}, there are quadratic characters $\eta$ of $\A_F^\times/F^\times$ matching with
$\lambda$ at places in $S\backslash \{v_0\}$ such that the following
global epsilon factors are 1:
$$\epsilon(\widetilde{\Pi} \otimes \eta) = \epsilon(\widetilde{\Pi} \otimes \eta\omega_{E/F}) =1.$$ 
We are then in the context of Conjecture \ref{conj}  which gives a character $\mu$ 
of $\A_F^\times/F^\times$ of order 2 matching with $\eta$ at all places of $S$
such that 
\begin{eqnarray*}
L(\frac{1}{2}, \widetilde{\Pi} \otimes \mu) &\not =&  0, \\
L(\frac{1}{2}, \widetilde{\Pi} \otimes \mu\omega_{E/F}) & \not = & 0.
\end{eqnarray*}

Since Lemma \ref{lemma8.1} also guarantees
\[\epsilon(\widetilde{\Pi}_v\otimes \eta_v) \epsilon(\widetilde{\Pi}_v\otimes \omega_v\eta_v)=
\epsilon(\widetilde{\Pi}_v \otimes \lambda_v) \epsilon(\widetilde{\Pi}_v\otimes \omega_v\lambda_v)\]
at each place $v$ of $F$, by the Saito-Tunnell theorem, 
we see that $\tilde{\Pi}_{v_0}$ is both $\eta_{v_0}$-distinguished and $\lambda_{v_0}$-distinguished 
even if $\eta_{v_0} \neq \lambda_{v_0}$. By the theorem of Waldspurger, the above non-vanishing of $L$-values 
then implies that $\widetilde{\Pi}$ is globally $\mu$-distinguished.
By local multiplicity one, this is enough to conclude that $\Pi^+$ is globally $\mu$-distinguished, provided we know that $\Pi^+_{v_0}$ is 
$\mu_{v_0}$-distinguished.  We prove that $\Pi^+_{v_0}$ is $\mu_{v_0}$-distinguished by 
proving that both the characters of $\frac{E_v^\times}{F_v^\times E_v^{\times 2}}$ appear on the same 
component of the restriction of $\tilde{\Pi}_v$ to $\GL_2^+(F_v)$ (so the possible difference between $\mu$ and $\lambda$ at $v_0$ 
has no consequence for distinction question), and this follows from the following lemma.
\end{proof}

\begin{lem}\label{Pra94}
 Let $E/F$ be a quadratic extension of $p$-adic fields with $p$ odd. Let $\pi^+$ be a supercuspidal representation of $\GL_2^+(F)$ of trivial central character which is distinguished with respect to $E^1$. Let $\tilde{\pi}$ be a supercuspidal representation of $\GL_2(F)$ 
which is distinguished with respect to $E^\times$ such that $\pi^+$ occurs in the restriction of $\tilde{\pi}$ to $\GL_2^+(F)$. Then the 
multiplicity of the trivial representation of $E^1$ in $\pi^+$ is $2$. 
\end{lem}

\begin{proof}
Note that since $p$ is odd, we have $E^\times/F^\times E^{\times 2} = \Z/2$, 
and both the characters of $E^\times/F^\times E^{\times 2}$ 
do appear in $\tilde{\pi}$ because of an application of the theorem of Saito-Tunnell in 
conjunction with  Lemma \ref{lemma8.3}. Thus if $\tilde{\pi}$ does not split into a direct sum of two representations, the assertion is obvious. So we assume that $\tilde{\pi}$ restricts to $\pi^+ \oplus \pi^-$ on $\GL_2^+(F)$. We need to show that $\pi^+$ is $\mu$-distinguished as well, where $\mu$ is the nontrivial character of $E^\times/F^\times E^{\times 2}$. In 
this case, $\tilde{\pi}$  corresponds to a monomial representation of $W_F$ of the form 
${\rm Ind}_{W_E}^{W_F} \chi$ for a character $\chi: E^\times \rightarrow \C^\times$ with $\chi|_{_{F^\times}}=\omega_{_{E/F}}$. By the extension of the Saito-Tunnell theorem to $\GL_2^+(F)$ due to the second author (cf. Theorem 1.2, \cite{dipendra}), what we need to show is that 
\[\epsilon(\chi \mu,\psi)=\epsilon(\chi,\psi),\]
where we take $\psi$ to be a nontrivial character of $E/F$. Note that we can take $\psi$ to be unramified, i.e., trivial on 
$\O_E$ but not on $\varpi_E^{-1} \O_E$. 

We will prove the above equality by making use of a theorem of Fr\"ohlich-Queyrut \cite{fq} --- according to which $\epsilon(\tau,\psi)=1$ if $\tau|_{_{F^\times}}=1$ ---
 as well as the behaviour of degree one epsilon factors under unramified 
character twists. In the following, $f(\chi)$ denotes the conductor of the multiplicative character $\chi$.

Since $\chi|_{_{F^\times}}=\omega_{_{E/F}}$, we have
\[\epsilon(\chi\widetilde{\omega},\psi)=1\]
by the theorem of Fr\"ohlich-Queyrut, where $\widetilde{\omega}$ denotes an extension of $\omega_{_{E/F}}$ to $E^\times$. 
Suppose $E/F$ is unramified. Then, it follows that
\[\epsilon(\chi,\psi)=(-1)^{f(\chi)},\]
as we can choose $\widetilde{\omega}$ to be unramified.
Similarly, we get
\[\epsilon(\chi\mu,\psi)=(-1)^{f(\chi\mu)}.\]
Thus if $f(\chi)>1$, then $f(\chi\mu)=f(\chi)$ and the equality of the epsilon factors follows. The case $f(\chi)=0$ does not arise since this would mean that $\chi=\chi^\sigma$ which is not possible since $\tilde{\pi}$ is supercuspidal. If $f(\chi)=1$, we claim that once again $f(\chi\mu)=f(\chi)$, as the only other option is $f(\chi\mu)=0$, and this also implies that $\chi=\chi^\sigma$ since $\mu = \mu^\sigma$ 
by the uniqueness of the quadratic character of $E^\times/F^\times E^{\times 2}$.

Now suppose $E/F$ is ramified. This forces $\mu$ to be unramified. Therefore,
\[\epsilon(\chi\mu,\psi)=(-1)^{f(\chi)}\epsilon(\chi,\psi),\]
and thus we only need to note that $f(\chi)$ is even by our assumptions. Indeed, $f(\chi)$ is either even or $1$ since $E/F$ is ramified and $\chi|_{_{F^\times}}=\omega_{_{E/F}}$, and $f(\chi)=1$ is ruled out when $q\equiv 1 {\rm ~ mod ~} 4$ since in this case, $\chi$ 
is forced to be Galois invariant, hence $\tilde{\pi}$ cannot be supercuspidal. Also, $q$ cannot be $3 {\rm ~ mod ~} 4$, since in that case $\tilde{\pi}$ is neither distinguished nor $\mu$-distinguished as can be seen by an application of the 
Saito-Tunnell theorem. 
\end{proof}

Note that the above proof goes through and proves an analogous lemma in the division algebra case except that at the very last step, when $E/F$ is ramified and $q \equiv 3 {\rm ~ mod ~} 4$, 
$\epsilon(\chi \mu,\psi)=-\epsilon(\chi,\psi)$ if $f(\chi)=1$. However, for purposes
of the local-global principle this is no problem: if a character $\chi$ of $F^\times$ thought of as a character of 
$E^\times$ through the norm mapping appears in a representation $\pi^+$ of $D^+$, 
then clearly so does $\chi\omega_{E/F}$ (being the same character of $E^\times$). For $q\equiv 3 {\rm ~ mod ~} 4$, 
 \[\frac{\epsilon(\widetilde{\pi} \otimes \chi)}{\chi(-1)}   =  -\frac{\epsilon(\widetilde{\pi} \otimes \chi\omega_{E/F})}{\chi(-1) 
\omega_{E/F}(-1)}.
\]  
This gives required change of sign argument used earlier to prove local-global principle for $E^1(\A_F)$ contained in $\SL_1(D)(\A_F)$,
assuming finitely many supersingular primes, Conjecture 1.3 and one odd prime where the representation is discrete series; we omit the details.

\vspace{2mm}

\noindent{\bf Remark:} To prove Theorem \ref{thm11.4} without the finiteness
condition on supersingular primes, we will need a finer version of 
Conjecture \ref{conj} which has allowed us the existence of the quadratic character 
$\eta$ at the end of this theorem. The refinement would seek to construct
$\eta$ with prescribed behaviour inside $S$, 
which is unramified at those places outside  $S$ 
where $\Pi$ is supersingular. This is because as we noted earlier, the 
behaviour of $\eta$ outside of $S$ and the 
supersingular primes does not matter for 
distinction questions as the representation $\Pi_v$ of $\GL_2(F_v)$ remains 
irreducible when restricted to $\GL_2^+(F_v)$. 
At least in the non-CM case, since the
supersingular set is rather `thin', one hopes that this strengthening
may be possible.

\section{A final remark}
The two cases of the local-global principle studied in the paper relied on the 
Asai lift and the base change map. One  part of the argument had 
to do with the fibers of these functorial maps. The other part consisted in 
proving that for $E/F$ a quadratic extension of number fields, certain characters of
$\A_E^\times$ whose restriction to $\A_F^\times$ is a Gr\"ossencharacter is itself a 
Gr\"ossencharacter, if we know certain properties of the character under basechange 
or Asai lift as the case may be. It seems worthy to isolate these as a question.
Before we do this, it must be added that at the moment, automorphy of the 
tensor product $\Pi \boxtimes \Pi'$, or of the Asai lift, is known only in certain cases,  
so either the questions below could be asked for only those cases, or we should be willing
to grant these in general.

\vspace{4mm}

\noindent{\bf Question 1:} Suppose $E$ is a number field, and $\Pi = \otimes \Pi_v$
is an irreducible admissible representation of $\GL_n(\A_E)$, and 
$\Pi' = \otimes \Pi'_v$ is an automorphic representation of $\GL_m(\A_E)$.
\begin{enumerate}

\item Suppose that  $\Pi \boxtimes \Pi'$ 
is automorphic. Then is there an 
automorphic representation  $\Pi''$ of  $\GL_n(\A_E)$ with $\Pi'' \boxtimes \Pi' 
= \Pi \boxtimes \Pi'$? What are the various automorphic representations $\Pi''$ 
of $\GL_n(\A_E)$ with this property? 
(This part of the question generalizes 
the notion of self-twists of automorphic representations.)

\item Suppose that  ${\rm BC}(\Pi \boxtimes \Pi')$ 
is automorphic. Then is there an 
automorphic representation  $\Pi''$ of  $\GL_n(\A_E)$ with ${\rm BC}(\Pi'' \boxtimes \Pi') 
= {\rm BC}(\Pi \boxtimes \Pi')$?

\item Suppose that  ${\rm As}(\Pi)$ 
is automorphic. Then is there an 
automorphic representation  $\Pi''$ of  $\GL_n(\A_E)$ with ${\rm As}(\Pi'' ) 
= {\rm As}(\Pi)$? What are the various automorphic representations $\Pi''$ 
of $\GL_n(\A_E)$ with this property? 
\end{enumerate}

\end{document}